\documentclass[hidelinks,onefignum,onetabnum]{siamart250211}

\usepackage{lipsum}
\usepackage{amsfonts}
\usepackage{graphicx}
\usepackage{epstopdf}
\usepackage{amsmath,amstext,amssymb}
\usepackage{mathrsfs}
\usepackage{booktabs}
\usepackage{stmaryrd}
\usepackage{algorithm}
\usepackage{algpseudocode}

\ifpdf
\DeclareGraphicsExtensions{.eps,.pdf,.png,.jpg}
\else
\DeclareGraphicsExtensions{.eps}
\fi

\newtheorem{example}{Example}
\newsiamremark{remark}{Remark}
\newsiamremark{assumption}{Assumption}
\crefname{assumption}{Assumption}{Assumption}
\newsiamthm{claim}{Claim}
\newsiamremark{fact}{Fact}
\crefname{fact}{Fact}{Facts}

\headers{Adaptive PIRaNNs with a PoU for solving PDEs}{Ran Bi and Weibing Deng}

\title{Adaptive Physics-informed Randomized Neural Networks with a Partition of Unity for Solving PDEs\thanks{Submitted to the editors DATE.
		\funding{This work was 
			partially supported by the NSF of China grant  12171237, and by the
			Ministry of Science and Technology of China grant 2020YFA0713800.}}}

\author{Ran Bi\thanks{School of Mathematics, Nanjing University, Nanjing 210093, People’s Republic of China
		{(\email{ranbi@smail.nju.edu.cn}, \email{wbdeng@nju.edu.cn}).}}\and Weibing Deng\footnotemark[2]
}

\usepackage{amsopn}

\ifpdf
\hypersetup{
  pdftitle={XI-PINN},
  pdfauthor={Ran Bi and Weibing Deng}
}
\fi

\definecolor{gold}{rgb}{0,0,1}

\begin{document}

\maketitle

\begin{abstract}
This paper establishes an approximation theory and develop an adaptive computational framework for randomized neural networks (RaNNs). For RaNNs of the form
$
\sum_{i=1}^{N} W_i \sigma(A_i\cdot x+B_i),
$
with hidden parameters uniformly sampled from a bounded set of scale $M$, we show that the choice
$
M \asymp N^{p/[2((p-1)(d+1)+p\eta)]}
$
yields an expected $W^{k,p}$-approximation error of order
$
\mathcal{O}\left(N^{-p\eta/[2((p-1)(d+1)+p\eta)]}\right),
$
where $\eta$ is associated with the regularity of the target function in a generalized Barron spectral space. This theoretical result demonstrates that lower regularity requires a larger sampling range for the hidden parameters. Motivated by the relationship between $M$ and $\eta$, we introduce an adaptive physics-informed RaNN (PIRaNN) method that which couples parameter sampling with an adaptive partition of unity via local affine scaling. Residual-based a posteriori error indicators and Dörfler marking are used to drive local refinement. Numerical experiments—including the 1D viscous Burgers' equation and 2D/3D problems with localized peaks and L-shaped corner singularities—demonstrate that our method preferentially refines regions exhibiting large gradients and singularities. Compared to standard non-adaptive RaNNs, the adaptive PIRaNN effectively captures complex local features with significantly enhanced accuracy and efficiency.
\end{abstract}



\section{Introduction}
In recent years, neural networks have garnered significant attention as a promising tool for solving partial differential equations (PDEs). {Unlike traditional numerical methods—such as the finite element and finite difference methods—which rely on spatial discretization over computational meshes, neural network-based approaches seek to approximate PDE solutions directly by training networks with strong universal approximation capabilities, thereby eliminating the need for mesh generation. Furthermore, these methods can effectively mitigate the curse of dimensionality often encountered in conventional mesh-based schemes.} As a result, a variety of neural network frameworks have been developed, including Physics-Informed Neural Networks (PINNs) \cite{raissi2019physics}, DeepONet \cite{lu2021learning}, and the Fourier Neural Operator (FNO) \cite{li2020fourier}, demonstrating considerable potential in overcoming the limitations of classical numerical techniques.

Theoretical research has extensively demonstrated the powerful function approximation capabilities of neural networks. For instance, constructive proofs of universal approximation theorems for deep neural networks under various activation functions have been provided in \cite{guhring2021approximation,de2021approximation,lu2021deep}. For shallow neural networks with simpler architectures, Ellacott \cite{ellacott1994aspects} showed that any function defined on a compact set can be approximated provided the activation function is non-polynomial. In a seminal work, Barron \cite{barron2002universal} employed statistical arguments to derive an approximation rate of $\mathcal{O}\left(N^{-1/2}\right)$ for shallow networks with a sigmoidal activation function, where $N$ denotes the number of neurons. Notably, this rate is independent of the input dimension $d$, offering theoretical support for circumventing the curse of dimensionality. This result was later refined by Klusowski and Barron \cite{klusowski2016uniform}, who improved the convergence rate to $\mathcal{O}\left(N^{-\frac12 - \frac{1}{d+1}}\right)$ using stratified sampling, thereby introducing dimension dependence. Subsequently, E et al. \cite{ma2022barron,ma2018priori} formalized the notion of Barron spaces, which characterize functions that can be efficiently approximated by shallow ReLU networks. Further extending this direction, Xu et al. \cite{siegel2020approximation,siegel2022high,xu2020finite} introduced Barron spectral spaces and derived approximation rates for shallow networks with more general activations. Most recently, Siegel and Xu \cite{siegel2024sharp} established sharp approximation bounds for shallow neural networks with the $\mathrm{ReLU}^k$ activation function, achieving an order of $\mathcal{O}\left(N^{-\frac12 - \frac{pk+1}{pd}}\right)$ in the $L^p$ norm. These theoretical advances underscore the significant potential of neural network-based approximations, particularly in high-dimensional settings.

{
	For PDE computations, optimizing all network parameters usually leads to a high-dimensional nonconvex problem. Numerous techniques have been proposed to mitigate this difficulty, including trainable activation scalings \cite{jagtap2020adaptive}, adaptive loss weighting \cite{wang2021understanding,wang2022and}, higher-order optimization methods \cite{wang2026gradient,rathore2024challenges}, and extended-variable formulations \cite{hu2022discontinuity,bi2025extended}. Randomized neural networks (RaNNs) take a different route: the hidden weights and biases are sampled once from a prescribed probability distribution and then held fixed, while only the output coefficients are trained. For a linear PDE, the resulting physics-informed problem is a linear least-squares problem. }This architecture is closely related to the extreme learning machine (ELM) \cite{huang2006extreme}. Its PDE variants include the physics-informed ELM \cite{dwivedi2020physics}, the domain-decomposed local ELM \cite{dong2021local}, and PoU-based random-feature methods \cite{chen2022bridging}. De et al. \cite{de2026least} imposed boundary conditions as hard equality constraints in the ELM least-squares formulation. Approximation and generalization properties of randomized architectures have also been studied in \cite{gonon2023random,guhring2021approximation,de2025approximation,neufeld2023universal,liu2025integral,zhang2024transferable}.

Since the internal parameters of RaNNs are fixed and sampled from a predetermined distribution, classical approximation theorems for general shallow neural networks are not directly applicable. In this work, motivated by the analysis in \cite{siegel2020approximation}, where approximation rates were established for shallow networks with polynomially decaying non-sigmoidal activations, we extend the definition of the Barron spectral space \cite{siegel2022high,siegel2020approximation} to a generalized Barron spectral space. We prove that functions belonging to this space can be approximated by RaNNs whose parameters are generated via uniform sampling from a bounded set, and we establish a corresponding convergence rate $\mathcal{O}\left( N^{-\frac{p \eta}{2[(p-1)(d+1) + p\eta]}}\right) $ in the Sobolev space $W^{k,p}$ for $2 \le p < \infty$. Our theoretical analysis demonstrates that the parameter $M$, governing the size of the sampling set for the inner-layer RaNN parameters, should satisfy $M \asymp N^{\frac{p}{2[(p-1)(d+1) + p\eta]}}$. Here, $\eta$ is associated with the regularity index of the generalized Barron spectral space, and $N$ represents the number of neurons (see Theorem~\ref{Theorem 1}). This scaling exposes an intrinsic dependence between the required sampling range and the regularity of the target function: for a fixed number of neurons, less smooth functions require a larger range to attain better approximation.

Leveraging the established relationship between the parameter sampling range and the target function's smoothness, we construct a PoU over the physical domain of the PDEs. The parameters of the RaNNs are first generated on a reference element and are then mapped to each physical subdomain via an affine transformation. This construction effectively couples the range of the sampled parameters to the local size of the PoU elements. In particular, a finer partition (smaller element size) corresponds to a locally enlarged effective sampling range for the RaNN parameters within that element. Motivated by strategies widely used in adaptive finite element methods \cite{dorfler1996convergent,cascon2008quasi,karakashian2007convergence}, we employ residual-based a posteriori error estimates to identify which elements require refinement. This leads to an adaptive  physics-informed randomized neural network (PIRaNN) framework, where the network parameter distribution and the physical discretization are co-adapted to efficiently resolve local solution features.
{
	The novelty and main contributions of this work are summarized as follows:
	\begin{itemize}
		\item[1.] We establish $W^{k,p}$ $(2\le p <\infty)$ approximation estimates for RaNNs whose hidden weights and biases are sampled uniformly from a bounded parameter set. The estimates separate the parameter-truncation error from the stochastic sampling error and analyze their dependence on the size of the sampling set, network size, spatial dimension, and the generalized Barron regularity of the target function. We further derive an improved convergence rate for stratified uniform sampling. 
		
		\item[2.] Guided by the established scaling relationship between the required sampling range $M$ and the optimal approximation error, we introduce an affine PoU construction that maps random features generated on a common reference element to local physical subdomains. Consequently, local refinement enlarges the effective hidden‑parameter range in physical coordinates, thus providing a theory‑guided mechanism to enhance the representation of localized solution features without requiring subdomain‑specific sampling distributions.
		
		\item[3.] We develop a residual-based adaptive PIRaNN algorithm using a posteriori error indicators and D\"orfler marking. The numerical experiments show that the method preferentially refines regions containing sharp localized peaks, corner singularities, and steep solution fronts, thereby enabling the RaNN to effectively capture local features.
	\end{itemize}
}
The remainder of this paper is organized as follows. Section~\ref{section 2} introduces the basic frameworks of RaNNs and PIRaNNs, followed by the definition of the proposed generalized Barron spectral space. Section~\ref{section 3} establishes a convergence analysis in Sobolev spaces for approximating functions in this generalized Barron space using RaNNs. Section~\ref{section 4} details the proposed adaptive PIRaNN algorithm. In Section~\ref{section 5}, comprehensive numerical experiments are presented to validate the theoretical analysis of Section~\ref{section 3} and to demonstrate the robust approximation capabilities of the adaptive PIRaNN method. Finally, conclusions are presented in Section~\ref{section 6}.

\section{Preliminaries}\label{section 2}
In this section, we first briefly introduce the framework of RaNNs and the PIRaNN method for solving PDEs. Then, to facilitate the subsequent theoretical analysis, we extend the definitions of the Barron spectral norms and the corresponding function spaces, motivated by the ideas in \cite{xu2020finite,siegel2020approximation}.

\subsection{Randomized neural networks}
We consider a single-hidden-layer feedforward network (SLFN) \cite{huang2006extreme} with randomly generated hidden weights. More specifically, given a probability distribution $\rho$, the random function can be defined as
\begin{equation}\label{SLFN}
	U_W^{A, B}(x) := \sum_{i=1}^{N} W_i \sigma\left( A_i \cdot x + B_i\right), \quad x\in \mathbb{R}^d,
\end{equation}
where $(A_1, B_1), ..., (A_N, B_N)$ are i.i.d. $\mathbb{R}^{d+1}$-valued random vectors drawn from $\rho$, $\sigma:\mathbb{R} \to \mathbb{R}$ is a fixed activation function and {the weights $W_1, ..., W_N \in \mathbb{C}$} can be chosen freely such that $U_W^{A, B}$ is a good approximation of the target function $u$. In fact, the SLFN can be interpreted as generating a family of basis functions $\left\lbrace \phi_i(x) = \sigma\left( A_i \cdot x + B_i\right), (A_i, B_i) \sim \rho \right\rbrace_{i=1}^N$ according to $\rho$. This family of functions spans a linear space $\mathcal{H}_{N} = \mathrm{span}\left\lbrace \phi_1, \phi_2, ..., \phi_N\right\rbrace$ whose properties critically depend on the choice of the probability distribution $\rho$. We note that one can also add an output bias $W_0$ to (\ref{SLFN}).
\subsection{PIRaNNs for solving PDEs}
The PIRaNN method combines the structure of RaNNs~\eqref{SLFN} with physics-based residual minimization principles. This approach circumvents the mesh dependency of traditional numerical methods and the complex nonconvex optimization challenges encountered in deep feedforward neural network training.

Consider a general PDE defined on a domain $\Omega \subset \mathbb{R}^d$ with boundary $\partial \Omega$ and $d \ge 1$:
\begin{equation} \label{PDEs}
	\left\lbrace 
	\begin{aligned}
		\mathcal{D}\left[ u\right](x) &= f(x), \quad &&\mathrm{in}\;\Omega,\\
		\mathcal{B}\left[ u\right](x) &= g(x), \quad &&\mathrm{on}\;\partial \Omega,
	\end{aligned}
	\right. 
\end{equation}
where $\mathcal{D}$ is a potentially linear (or nonlinear) differential operator and $\mathcal{B}$ represents the boundary operator. The PIRaNN approach seeks the optimal parameters $W^*=\left\lbrace W_i^*\right\rbrace_{i=1}^N$ such that $U_{W^*}^{A,B}$ approximates the solution of (\ref{PDEs}). This is achieved using a collocation method based on residual minimization of (\ref{PDEs}). Let $K_\Omega$ and $K_{\partial \Omega}$ denote the numbers of collocation points in $\Omega$ and $\partial \Omega$, respectively. We define two sets of collocation points:
\begin{equation*}
	C_\Omega = \left\lbrace x_{\Omega, j} \right\rbrace_{j=1}^{K_\Omega} \subset \Omega, \quad C_{\partial \Omega} = \left\lbrace x_{\partial\Omega, j} \right\rbrace_{j=1}^{K_{\partial \Omega}} \subset \partial \Omega.
\end{equation*}
By enforcing (\ref{PDEs}) at each collocation point, we define the corresponding loss function as follows:
\begin{equation*}
	\begin{aligned}
		\mathcal{L}(W) &= \sum_{j=1}^{K_\Omega} \lambda_{\Omega,j}^2 \left| \mathcal{D}[U_W^{A,B}](x_{\Omega, j}) - f(x_{\Omega, j})\right|^2 \\
		&\quad + \sum_{j=1}^{K_{\partial\Omega}} \lambda_{\partial\Omega,j}^2 \left| \mathcal{B}[U_W^{A,B}](x_{\partial\Omega, j}) - g(x_{\partial\Omega, j})\right|^2,
	\end{aligned}
\end{equation*}
where $\left\lbrace \lambda_{\Omega, j} \right\rbrace$ and $\left\lbrace \lambda_{\partial \Omega, j} \right\rbrace$ are the corresponding weight parameters. 

When $\mathcal{D}$ is a linear operator and the basis functions $\left\lbrace \phi_i \right\rbrace_{i=1}^N$ are linearly independent, we can obtain the solution $U_{W^*}^{A,B} = \sum_{i=1}^N W^*_i \sigma(A_i \cdot x + B_i)$ by
\begin{equation} \label{eq 2.3}
	W^* = H^\dag T,
\end{equation}
where
\begin{equation*}
	H = 
	\begin{bmatrix}
		\left[ \lambda_{\Omega, j} \mathcal{D}\left[ \sigma \left( A_i \cdot x_{\Omega, j} + B_i\right) \right] \right]_{K_\Omega \times N} \\
		\left[ \lambda_{\partial\Omega, j}\mathcal{B}\left[ \sigma \left( A_i \cdot x_{\partial \Omega, j} + B_i\right) \right] \right]_{K_{\partial \Omega} \times N}
	\end{bmatrix},
	\quad
	T =
	\begin{bmatrix}
		\left[\lambda_{\Omega, j} f(x_{\Omega, j}) \right]_{K_\Omega \times 1} \\
		\left[\lambda_{\partial\Omega, j} g(x_{\partial \Omega, j}) \right]_{K_{\partial \Omega} \times 1}
	\end{bmatrix},
\end{equation*}
and $H^\dag$ is the Moore-Penrose generalized inverse of $H$. Similarly, when $\mathcal{D}$ is a nonlinear differential operator, one can formulate a nonlinear least-squares problem (see \cite{dong2021local}).

\subsection{Generalized Barron spectral spaces}
For a real-valued function $f$ defined on a bounded domain $\Omega \subset \mathbb{R}^d$, its approximation by single-hidden-layer neural networks was first characterized by Barron using a Fourier representation of $f$ under suitable smoothness assumptions, thereby establishing a theoretical approximation rate \cite{barron2002universal,klusowski2016risk}. Subsequently, the Barron spectral space was defined by Xu et al. \cite{xu2020finite,siegel2022high} to characterize the regularity (or smoothness) of functions. Consider all extensions $f_e: \mathbb{R}^d \to \mathbb{R}$ and define the Barron spectral norm for $s \ge 1$:
\begin{equation} \label{Barron spectral norm}
	\|f\|_{\mathcal{B}_s(\Omega)} = \mathop{\inf}_{f_e|_\Omega = f} \int_{\mathbb{R}^d} \left( 1 + |\xi|\right)^s|\widehat{f_e}(\xi)| \mathrm{d}\xi, 
\end{equation}
and the Barron spectral space
\begin{equation*}
	\mathcal{B}_{s}(\Omega) = \left\lbrace f: \Omega \to \mathbb{R} \big|  \|f\|_{\mathcal{B}_s(\Omega)} < \infty \right\rbrace.
\end{equation*}
The Barron spectral norm \eqref{Barron spectral norm} can be interpreted as a weighted $L^1$ norm of the Fourier transform $\widehat{f_e}$ in the frequency domain, which effectively describes the decay behavior of $\widehat{f_e}$ and hence the smoothness of the original function $f$. To support the subsequent analysis of the approximation properties of RaNNs, we now introduce a generalized version of the Barron spectral norm. For $k \ge 0$, $s \ge 1$, $p \in [1, \infty)$ and a multi-index $\alpha=(\alpha_1, \alpha_2, \dots, \alpha_d)$ with $\alpha_i \ge 0$ and $|\alpha| = \sum_{i=1}^d \alpha_i$, we define
\begin{equation}\label{Generalized spectral norm}
	\|f\|_{\mathcal{B}_s^{k,p}(\Omega)} = \mathop{\inf}_{f_e|_\Omega = f} \left( \sum_{|\alpha| \le k} \int_{\mathbb{R}^d} \left( \left( 1 + |\xi| \right)^s \left| \partial^\alpha_\xi \widehat{f_e}(\xi) \right| \right)^p \mathrm{d}\xi \right)^{1/p}.
\end{equation}
The corresponding generalized Barron spectral space is
{
	\begin{equation*}
		\mathcal{B}_s^{k,p}(\Omega) = \left\lbrace f: \Omega \to \mathbb{C} \big|  \|f\|_{\mathcal{B}_s^{k,p}(\Omega)} < \infty \right\rbrace.
	\end{equation*}
}For simplicity, we write $\mathcal{B}^p_s(\Omega) = \mathcal{B}^{0,p}_s(\Omega)$ when $k=0$. It can be observed that the generalized Barron spectral norm corresponds to the weighted Sobolev $W^{k,p}$ norm of $\widehat{f_e}$ in the frequency domain. Similarly to the Barron spectral norm, the norm defined in (\ref{Generalized spectral norm}) can characterize the smoothness of the function $f$. The key difference lies in its dependence on the derivatives of $\widehat{f_e}$. From the properties of the Fourier transform \cite{folland1999real}, we recall that when $x^\alpha f_e \in L^1(\mathbb{R}^d)$, the identity $\partial_\xi^\alpha\widehat{f_e} = \widehat{(-i x)^\alpha f_e}$ holds. This implies that the extension $f_e$ must satisfy more stringent decay conditions; i.e., $x^\alpha f_e \in L^1(\mathbb{R}^d)$.

\section{Approximation properties for RaNNs in Sobolev norms} \label{section 3}
Since our objective is to solve PDEs using RaNNs, it becomes necessary to characterize the approximation of {the unknown solution $u: \Omega \to \mathbb{C}$} and its derivatives by RaNNs---specifically, their approximation properties in Sobolev spaces. This requires establishing a theoretical connection between the Sobolev space and our defined generalized Barron spectral space. For simplicity, in this section, we use the shorthand notation $A \lesssim B$ for the inequality $A \le CB$, where $C > 0$ is a generic constant independent of the number of neurons in the RaNNs.

Our analysis is based on the Fourier representation of the solution function $u$, and we require the activation function $\sigma$ to have localized properties (in contrast to commonly used globally defined activation functions). We thus need the following assumptions:

\begin{assumption} \label{assumption 1}
	{Let $u_e: \mathbb{R}^d \to \mathbb{C}$ be a global extension of $u$} such that $u_e \in L^1(\mathbb{R}^d)$ and its Fourier transform $\widehat{u_e} \in L^1(\mathbb{R}^d)$.
\end{assumption}

\begin{assumption} \label{assumption 2}
	Let {$k \ge 0$} be an integer and $r > 1$ be a real number. The activation function {$\sigma \in W^{k, \infty}(\mathbb{R})$} is assumed to be nonzero. Furthermore, for every integer {$\nu$} with {$0 \le \nu \le k$}, {$\sigma^{(\nu)}$} satisfies the polynomial decay condition
	{
		\begin{equation}
			|\sigma^{(\nu)}(x)| \le C_r (1 + |x|)^{-r},
		\end{equation}
	}
	where $C_r > 0$ is a constant depending on $r$, $\sigma$ and $\nu$.
\end{assumption}

\begin{remark} \label{remark 1}
	A variety of activation functions satisfy Assumption~\ref{assumption 2}, such as the Gaussian kernel function and the function $\tanh(x + 0.5) - \tanh(x - 0.5)$ constructed as the difference of two shifted $\tanh$ activations.
\end{remark}

Under the above assumptions, Siegel and Xu \cite{siegel2020approximation} established an approximation result for SLFNs. Specifically, let $\Omega \subset \mathbb{R}^d$ be a bounded domain and $\varepsilon > 0$. Consider the class of functions representable by a network with $N$ hidden neurons, 
\begin{equation*}
	\Sigma_N^d(\sigma) = \left\lbrace \sum_{i=1}^N w_i \sigma(\xi_i \cdot x + s_i) \big|  \xi_i \in \mathbb{R}^d, s_i,w_i \in \mathbb{R} \right\rbrace.
\end{equation*}
If the activation function $\sigma$ satisfies Assumption~\ref{assumption 2}, then for any function $u$ belonging to the space $\mathcal{B}_{k+1+\varepsilon}^1(\Omega)$ and fulfilling Assumption~\ref{assumption 1}, the following error estimate holds:
\begin{equation} \label{xu approx thm}
	\inf_{u_N \in \Sigma_N^d(\sigma)}  \|u-u_N\|_{H^k(\Omega)} \lesssim |\Omega|^{\frac{1}{2}} N^{-\frac{1}{2} - \frac{\eta}{(2+\eta)(d+1)}} \|u\|_{\mathcal{B}_{k+1+\varepsilon}^1(\Omega)},
\end{equation}
where $\eta = \min(r-1, \varepsilon)$.
%

However, this approximation result cannot be directly applied to the RaNNs defined in (\ref{SLFN}). Although RaNNs are structurally similar to single-hidden-layer neural networks, a crucial difference lies in the treatment of internal parameters: in RaNNs, the weights and biases of the hidden layer are randomly sampled from a prescribed probability distribution $\rho$ and remain fixed during training, whereas in conventional networks these parameters are optimized. Moreover, the proof of \eqref{xu approx thm} relies on the existence of a probability distribution over parameters that is linked to the Barron spectral norm $\|u\|_{\mathcal{B}_s^1}$. In practice, however, the smoothness properties of the target function $u$ are often unknown, making it difficult to design such a distribution a priori. In the RaNN setting, the most commonly used sampling distributions are uniform and Gaussian. In what follows, we present an approximation theorem for RaNNs based on uniform sampling.

Under Assumption~\ref{assumption 2}, the activation function $\sigma$ belongs to $L^1(\mathbb{R})$. Consequently, its Fourier transform $\widehat{\sigma}$ is well-defined and continuous, and there exists some $a \neq 0$ such that $\widehat{\sigma}(a) \neq 0$. Without loss of generality, we assume $a > 0$ and $a = \mathcal{O}(1)$. By a change of variables, we obtain the representation
\begin{equation} \label{eq 3.2}
	e^{ia\xi \cdot x} = \dfrac{1}{\widehat{\sigma}(a)} \int_\mathbb{R} \sigma(\xi \cdot x + s) e^{-ias} \mathrm{d}s.
\end{equation}
Applying the Fourier inversion theorem together with Assumption~\ref{assumption 1}, the following identity holds almost everywhere in $\Omega$:

\begin{equation} \label{eq 3.3}
	\begin{aligned}
		u(x) = u_e(x)
		= \dfrac{1}{(2\pi)^d} \int_{\mathbb{R}^d} e^{i\xi \cdot x} \widehat{u_e}(\xi) \mathrm{d}\xi.
	\end{aligned}
\end{equation}

To quantify the approximation error of RaNNs, we first establish several auxiliary lemmas.
\begin{lemma} \label{lemma 1}
	Let $\Omega \subset \mathbb{R}^d$ be a bounded domain and let the activation function $\sigma$ satisfy Assumption~\ref{assumption 2}. Then for any $l \ge 1$, the following integral estimate holds:
	\begin{equation*}
		\int_{\mathbb{R}}\left| \sigma^{({\nu})} \left( \xi \cdot x + s\right) \right|^l \mathrm{d}s \le C R \left( 1 + \left| a\xi \right|\right),
	\end{equation*}
	where $R = \max_{x\in\Omega} |x| \ge 1$ and $C > 0$ is a constant.
\end{lemma}
\begin{proof}
	By the triangle inequality and the boundedness of $\Omega$, we have
	\begin{equation*}
		\left| \xi \cdot x + s\right| \ge \max \left( 0, |s| - R|\xi|\right).
	\end{equation*}
	Combining this with Assumption~\ref{assumption 2} yields
	\begin{equation*}
		\left|\sigma^{({\nu})} \left( \xi \cdot x + s\right) \right| \le C_r \left( 1 + \max \left( 0, |s| - R|\xi|\right) \right)^{-r}.
	\end{equation*}
	For simplicity, we denote $
	h(\xi ,s) = \left( 1 + \max \left( 0, |s| - R|\xi|\right) \right)^{-r}.$
	The decay rate of the function $h$ is sufficiently fast for it to be integrable in $s$. For $l \ge 1$, the following estimate holds:
	\begin{equation*}
		\begin{aligned}
			\int_{\mathbb{R}} \left|\sigma^{({\nu})} \left( \xi \cdot x + s\right) \right|^l \mathrm{d} s
			&\le C_r^l \int_{\mathbb{R}} h(\xi, s)^l \mathrm{d} s \\
			&= C_r^l \left( \int_{|s| \le R|\xi|} 1 \, \mathrm{d}s + \int_{|s| > R|\xi|} \left( 1 + |s| - R|\xi|\right)^{-lr} \mathrm{d}s \right) \\
			&\le C_r^l \left( 2 R |\xi| + \frac{2}{lr-1} \right) \le C(l, r, a) R \left( 1 + |a\xi|\right).
		\end{aligned}
	\end{equation*}
	This completes the proof.
\end{proof}

{
	Define the bounded parameter sets
	\begin{equation*}
		G_\xi(M) := \left\lbrace \xi \in \mathbb{R}^d \big| |\xi| \le \dfrac{M}{2 R}\right\rbrace,\quad
		G_s(M) := \left\lbrace s \in \mathbb{R} \big| |s| \le M\right\rbrace,
	\end{equation*}
	where $M>0$ is a constant to be chosen later.} In the following, we assume $R = \mathcal{O}(1)$. By translating $\Omega$ if necessary, we may assume that the origin is located at the center of $\Omega$. Consequently, $R$ is comparable to the radius of $\Omega$ and $|\Omega| = \mathcal{O}(1)$.

\begin{lemma} \label{lemma 2}
	Suppose that Assumptions~\ref{assumption 1} and~\ref{assumption 2} hold. Let $u \in \mathcal{B}_{k+m}^p(\Omega)$ for $p \ge 2$, $m \ge 0$, and an integer $k \ge 0$, where $m > \frac{p-1}{p}d$. 
	Define the truncated function $u^M$ by restricting the integration domain of the parameters to $G_{\xi}(M) \times G_{s}(M)$:
	\begin{equation} \label{eq 3.4}
		u^M(x) := \int_{G_\xi(M)} \int_{G_s(M)} \beta(\xi, s) \sigma(\xi \cdot x + s) \mathrm{d}s \mathrm{d}\xi,
	\end{equation}
	where
	\begin{equation*}
		\beta(\xi, s) =  \dfrac{a^d}{(2\pi)^d \widehat{\sigma}(a)} \widehat{u_e}(a \xi) e^{-ias}.
	\end{equation*}
	Then, 
	the following estimate holds:
	\begin{equation}\label{estma1}
		\left\| u - u^M\right\|_{W^{k,p}(\Omega)} \lesssim M^{-\eta} \left\| u\right\|_{\mathcal{B}_{k+m}^p(\Omega)},
	\end{equation}
	where $\eta = \min\left( m - \frac{p-1}{p} d, r-1\right)$.
\end{lemma}
\begin{proof}
	We first introduce a band-limited approximation of $u_e$ defined by
	\begin{equation} \label{eq 3.6}
		u_e^M := \dfrac{1}{(2 \pi)^d} \int_{|\xi|\le\frac{aM}{2R}} e^{i\xi \cdot x} \widehat{u_e}(\xi) \mathrm{d}\xi.
	\end{equation}
	The proof proceeds by estimating the $W^{k,p}(\Omega)$ error in two steps as follows.  
	
	\paragraph{Step 1: Estimate of $\left\| u - u_e^M\right\|_{W^{k,p}(\Omega)}$.}
	
	For any multi-index $\alpha$ with $|\alpha| \le k$, it follows from \eqref{eq 3.3} and \eqref{eq 3.6} that
	\begin{equation*}
		\partial_x^\alpha\left(u_e(x)-u_e^M(x)\right) = \dfrac{1}{(2 \pi)^d} \int_{|\xi|>\frac{aM}{2R}} \Big(\prod_{i=1}^d (i\xi_i)^{\alpha_i}\Big) e^{i\xi \cdot x} \widehat{u_e}(\xi) \mathrm{d}\xi.
	\end{equation*}
	Applying H\"older's inequality, we obtain
	\begin{equation*}
		\begin{aligned}
			| \partial_x^\alpha ( u_e(x)& - u_e^M(x)) |^p 
			{\lesssim \left(\int_{|\xi|>\frac{aM}{2R}}(1 + |\xi|)^{|\alpha|} \dfrac{\left| \xi \right|^m}{\left| \xi \right|^m} \left| \widehat{u_e}(\xi)\right|  \mathrm{d}\xi \right)^p} \\
			& \le \left(\int_{|\xi|>\frac{aM}{2R}}\left| \xi \right|^{-qm} \mathrm{d}\xi \right)^{\frac{p}{q}} \int_{|\xi|>\frac{aM}{2R}} \left( 1+|\xi|\right)^{p\left( |\alpha|+m\right) } \left| \widehat{u_e}(\xi)\right|^p \mathrm{d}\xi  \\
			& \le \left(\int_{|\xi|>\frac{aM}{2R}}\left| \xi \right|^{-qm}\mathrm{d}\xi \right)^{\frac{p}{q}} \|u\|_{\mathcal{B}_{|\alpha|+m}^p(\Omega)}^p,
		\end{aligned}
	\end{equation*}
	where $q$ is the conjugate exponent of $p$ (i.e.\ $1/p + 1/q = 1$). For $m > \frac{d}{q} = \frac{p-1}{p}d$, we have
	\begin{equation*}
		\begin{aligned}
			\int_{|\xi| > \frac{aM}{2R}} |\xi|^{-qm} \mathrm{d}\xi 
			= \omega_d \int_{\frac{aM}{2R}}^{\infty} r^{-qm+d-1} \mathrm{d}r 
			\lesssim M^{-qm+d}.
		\end{aligned}
	\end{equation*}
	Consequently, summing over all $|\alpha|\le k$ yields the desired estimate in $W^{k,p}(\Omega)$:
	\begin{equation} \label{ieq 3.7} 
		\left\| u - u_e^M\right\|_{W^{k,p}(\Omega)} \lesssim M^{-m + \frac{p-1}{p}d} \|u\|_{\mathcal{B}_{k+m}^p(\Omega)}.
	\end{equation}
	
	\paragraph{Step 2: Estimate of $\left\| u^M - u_e^M\right\|_{W^{k,p}(\Omega)}$.}
	
	Using the representation \eqref{eq 3.2} in the definition of $u_e^M$, we can rewrite $u_e^M$ in a form similar to $u^M$:
	\begin{equation*}
		\begin{aligned}
			u_e^M &=\int_{|\xi| \le \frac{M}{2R}} \int_{\mathbb{R}} \beta(\xi, s) \sigma\left( \xi \cdot x + s\right) \mathrm{d}s \mathrm{d}\xi.
		\end{aligned}
	\end{equation*}
	Consequently, the difference $u_e^M - u^M$ involves only the tail of the $s$-integration:
	\begin{equation*}
		u_e^M(x)-u^M(x) = \int_{|\xi| \le \frac{M}{2R}} \int_{|s| > M} \beta(\xi, s) \sigma(\xi \cdot x + s) \mathrm{d}s \mathrm{d}\xi.
	\end{equation*}
	For any multi-index $\alpha$ with $|\alpha| \le k$, we have
	\begin{equation*}
		\begin{aligned}
			\left| \partial_x^\alpha \left( u^M(x) - u_e^M(x)\right) \right|^p \lesssim \left( \int_{|\xi| \le \frac{M}{2R}} \int_{|s| > M} |a\xi|^{|\alpha|} |\widehat{u_e}(a \xi)|\left|\sigma^{(|\alpha|)}\left(\xi \cdot x + s \right) \right| \mathrm{d}s \mathrm{d}\xi \right)^p.
		\end{aligned}
	\end{equation*}
	Notice that in the integration domain of the above equation, we have $|s| - R|\xi| \ge M/2$.
	Using the decay property of $\sigma$, it follows that
	\begin{equation*}
		\begin{aligned}
			\int_{|s| > M} \left| \sigma^{(|\alpha|)}(\xi \cdot x + s) \right| \mathrm{d}s
			\lesssim \int_{|s| > M} \left( 1 + |s| - R|\xi|\right)^{-r} \mathrm{d}s 
			\lesssim (1 + M/2)^{1-r}.
		\end{aligned}
	\end{equation*}
	Thus, for $M \ge 2$, we have
	\begin{equation*}
		\int_{|s| > M} \left| \sigma^{(|\alpha|)}(\xi \cdot x + s) \right|\mathrm{d}s \lesssim M^{1-r}.
	\end{equation*}
	Therefore, it follows from H\"older's inequality and the fact that $m> \frac{d}{q}$ that
	\begin{equation*}
		\begin{aligned}
			| \partial_x^\alpha ( u^M(x)& - u_e^M(x)) |^p 
			\lesssim  M^{p(1-r)} \left( \int_{|\xi| \le \frac{M}{2R}} |a\xi|^{|\alpha|} |\widehat{u_e}(a \xi)| \mathrm{d}\xi\right)^p \\
			&\lesssim  M^{p(1-r)} \left( \int_{|\xi| \le \frac{M}{2R}} \dfrac{(1 + |a\xi|)^{m}}{(1 + |a\xi|)^{m}} \left( 1+|a\xi|\right) ^{|\alpha|}|\widehat{u_e}(a \xi)| \mathrm{d}\xi\right)^p \\
			& \lesssim  M^{p(1-r)} \left( \int_{\mathbb{R}^d} \left( 1+|a\xi|\right) ^{-qm} \mathrm{d}\xi \right)^{p/q} \|u\|_{\mathcal{B}_{|\alpha|+m}^p(\Omega)}^p\\
			&\lesssim  M^{p(1-r)} \|u\|_{\mathcal{B}_{|\alpha|+m}^p(\Omega)}^p.
		\end{aligned}
	\end{equation*}
	Summing over all $|\alpha|\le k$, we obtain
	\begin{equation} \label{ieq 3.8}
		\left\| u^M - u_e^M\right\|_{W^{k,p}(\Omega)} \lesssim M^{1-r} \|u\|_{\mathcal{B}_{k+m}^p(\Omega)}.
	\end{equation}
	Thus, combining \eqref{ieq 3.7} and \eqref{ieq 3.8} via the triangle inequality, we conclude the estimate \eqref{estma1} immediately.
\end{proof}

\begin{lemma} \label{lemma 3}
	Let $(\Omega, \mathcal{F}, \mu)$ be a bounded measure space and $(\mathcal{D}, \mathcal{B}, \mathcal{P})$ be a probability space. {Let $X_i:\mathcal{D} \times \Omega \to \mathbb{C}, (i = 1, \dots, N)$ be i.i.d. random fields} which are measurable with respect to the product $\sigma$-algebra $\mathcal{B} \otimes \mathcal{F}$. Assume that for each $x \in \Omega$, {the maps $X_i(\cdot, x): \mathcal{D} \to \mathbb{C}, i=1,\dots,N$ are i.i.d. random variables} on $(\mathcal{D}, \mathcal{B}, \mathcal{P})$, and for $p \in [2, \infty)$, $\mathbb{E} \left[ \|X_1\|_{L^p(\Omega)}^p \right] < \infty$, where $\|X_1\|_{L^p(\Omega)}^p = \int_\Omega |X_1(\cdot, x)|^p \, \mathrm{d}\mu(x)$.
	Then the following estimate holds:
	\begin{equation*}
		\mathbb{E} \left[ \left\| \mathbb{E} \left[ X_1 \right] - \dfrac{1}{N} \mathop{\sum}_{i=1}^N X_i \right\|_{L^p(\Omega)} \right] \le \dfrac{ C_p}{N^{1/2}} \left( \int_\Omega \mathbb{E} \left[ \left| \mathbb{E}\left[ X_1\right]  - X_1\right| ^p\right]  \mathrm{d}\mu(x) \right)^{\frac{1}{p}},
	\end{equation*}
	where $C_p>0$ depends only on $p$.
\end{lemma}

The proof of Lemma~\ref{lemma 3} can be found in Appendix~\ref{appendix A}. Based on the lemmas and assumptions established above, we now establish the following approximation theorem:

\begin{theorem} \label{Theorem 1}
	Suppose that Assumptions~\ref{assumption 1} and~\ref{assumption 2} hold, and assume that the target function $u \in \mathcal{B}_{k+m}^{p}(\Omega)$ for $p \ge 2$, $m \ge 0$ and an integer $k \ge 0$, with $m > \frac{p-1}{p}d$. Then, when the truncation parameter {$M \asymp  N^{\frac{p}{2[(p-1)(d+1) + p\eta]}}$,}
	the following error estimate holds:
	\begin{equation}\label{mainerror}
		\mathbb{E}\left[ \| u - U_W^{A,B} \|_{W^{k,p}(\Omega)} \right] \lesssim \|u\|_{\mathcal{B}_{k+m}^p(\Omega)} \, N^{-\frac{p \eta}{2[(p-1)(d+1) + p\eta]}},
	\end{equation}
	where $U_W^{A,B}$ is defined in (\ref{SLFN}) with parameters $(A_i, B_i)$, $i=1,\dots,N$, sampled i.i.d. from the uniform distribution over the bounded set $G_\xi(M) \times G_s(M)$, and $\eta = \min\left( m - \frac{p-1}{p} d, r-1\right)$.
\end{theorem}

\begin{proof}    
	First, we reformulate the truncated function $u^M$ defined in \eqref{eq 3.4} as an expectation with respect to a probability measure:
	\begin{equation} \label{eq 3.10}
		u^M(x) = \mathbb{E}_{\widetilde{\pi}} \left[ f(\xi, s) \sigma(\xi \cdot x + s) \right],
	\end{equation}
	where $\widetilde{\pi}$ denotes the uniform probability measure on the set $G_\xi(M) \times G_s(M)$, $\pi_M$ is the corresponding uniform probability density function, i.e.,
	{
		\begin{equation*}
			\mathrm{d} \widetilde{\pi}(\xi, s) = \pi_M \mathrm{d}\xi \mathrm{d}s\ \text{with}\  \pi_M=\frac{1}{|G_\xi(M)|\,|G_s(M)|} = \frac{2^{d-1}\Gamma(d/2 + 1)}{\pi^{d/2}} \frac{R^d}{M^{d+1}},
		\end{equation*}
	}
	and the coefficient function is given by
	\begin{equation*}
		f(\xi, s) = \beta(\xi, s) / \pi_M.
	\end{equation*}
	Define the random vector
	\begin{equation*}
		W = (W_1, \dots, W_N), \quad W_i = \frac{1}{N} f(A_i, B_i).
	\end{equation*}
	Then the corresponding RaNN function is given by
	\begin{equation*}
		U_W^{A,B}(x) = \sum_{i=1}^{N} W_i \sigma(A_i \cdot x + B_i).
	\end{equation*}
	For each fixed $x \in \Omega$, define the random variables
	\begin{equation*}
		X_i(x) := f(A_i, B_i) \sigma(A_i \cdot x + B_i), \quad i=1,\dots,N.
	\end{equation*}
	For all multi-indices $\alpha$ with $|\alpha| \le k$, the $X_i(x)$ are i.i.d. and satisfy
	\begin{equation*}
		\mathbb{E}\left[ \partial_x^\alpha X_i(x) \right] = \partial_x^\alpha u^M(x).
	\end{equation*}
	Then it follows from Lemma~\ref{lemma 3} that
	\begin{equation} \label{ieq 3.11}
		\begin{aligned}
			\mathbb{E}\left[ \left\| \partial_x^\alpha u^M - \partial_x^\alpha U_W^{A,B} \right\|_{L^p(\Omega)} \right] 
			&\le \frac{C_p}{N^{1/2}} \biggl( \int_\Omega \mathbb{E} \left[ \left| \partial_x^\alpha u^M - \partial_x^\alpha X_1 \right|^p \right] \mathrm{d}x \biggr)^{1/p}.
		\end{aligned}
	\end{equation}
	To bound the right-hand side, we first estimate the pointwise expectation. By Jensen's inequality, we have
	\begin{equation}\label{pointerr}
		\begin{aligned}
			\mathbb{E} \left[ \left| \partial_x^\alpha u^M - \partial_x^\alpha X_1 \right|^p \right]
			&\le 2^{p-1} \left( \mathbb{E} \left[ |\partial_x^\alpha u^M|^p \right] + \mathbb{E} \left[ |\partial_x^\alpha X_1|^p \right] \right) \\
			&= 2^{p-1} \left( \left| \mathbb{E} [\partial_x^\alpha X_1] \right|^p + \mathbb{E} \left[ |\partial_x^\alpha X_1|^p \right] \right) \\
			&\le 2^{p} \, \mathbb{E} \left[ \left| \partial_x^\alpha X_1 \right|^p \right].
		\end{aligned}
	\end{equation}
	Further, from Lemma~\ref{lemma 1} with $l=p$ and the fact that $1 < pm$, it follows that 
	\begin{equation*}
		\small
		\begin{aligned}
			\mathbb{E} \left[ \left| \partial_x^\alpha X_1 \right|^p \right]
			&= \mathbb{E} \left[ \left| f(\xi, s) \, \partial_x^\alpha \sigma(\xi \cdot x + s) \right|^p \right] \\
			&= \mathbb{E} \left[ \left| f(\xi, s) \left( \prod_{i=1}^d a^{-\alpha_i} (a \xi_i)^{\alpha_i} \right) \sigma^{(|\alpha|)}(\xi \cdot x + s) \right|^p \right] \\
			&\lesssim \int_{G_\xi(M)} \int_{G_s(M)} \left| f(\xi, s) (1 + |a \xi|)^{|\alpha|} \sigma^{(|\alpha|)}(\xi \cdot x + s) \right|^p \pi_M \,\mathrm{d}s \,\mathrm{d}\xi \\
			&= \pi_M^{1-p}\int_{G_\xi(M)} \int_{G_s(M)} \left| (1 + |a \xi|)^{|\alpha|} \beta(\xi, s) \sigma^{(|\alpha|)}(\xi \cdot x + s) \right|^p \,\mathrm{d}s \,\mathrm{d}\xi \\
			&\lesssim {\pi_M^{1-p} \int_{G_\xi(M)} (1 + |a\xi|)^{p|\alpha|+1} |\widehat{u_e}(a\xi)|^p \int_{G_s(M)}\left| \sigma^{(|\alpha|) }\left( \xi \cdot x + s\right)\right|^p \mathrm{d}s \, \mathrm{d}\xi}\\
			&\lesssim {\pi_M^{1-p}\int_{G_\xi(M)} (1 + |a\xi|)^{p|\alpha|+pm} |\widehat{u_e}(a\xi)|^p\,\mathrm{d}\xi }\\
			&\lesssim M^{(p-1)(d+1)} \|u\|_{\mathcal{B}^p_{|\alpha|+m}(\Omega)}^p.
		\end{aligned}
	\end{equation*}
	Therefore, integrating \eqref{pointerr} over $\Omega$ and taking the $p$-th root gives
	\begin{equation*}
		\biggl( \int_\Omega \mathbb{E} \left[ \left| \partial_x^\alpha u^M - \partial_x^\alpha X_1 \right|^p \right] \mathrm{d}x \biggr)^{1/p} \lesssim M^{\frac{(p-1)(d+1)}{p}} \|u\|_{\mathcal{B}^p_{|\alpha|+m}(\Omega)}.
	\end{equation*}
	Thus, summing over all multi-indices with $|\alpha| \le k$ and combining with \eqref{ieq 3.11} yields
	\begin{equation*}
		\mathbb{E}\left[ \left\| u^M - U_W^{A,B} \right\|_{W^{k,p}(\Omega)} \right] \lesssim M^{\frac{(p-1)(d+1)}{p}} N^{-\frac{1}{2}} \|u\|_{\mathcal{B}^p_{k+m}(\Omega)}.
	\end{equation*}
	Finally, applying Minkowski's inequality and Lemma~\ref{lemma 2} yields the estimate
	\begin{equation*}
		\begin{aligned}
			\mathbb{E}\left[ \left\| u - U_W^{A,B} \right\|_{W^{k,p}(\Omega)} \right] 
			&\le \| u - u^M \|_{W^{k,p}(\Omega)} + \mathbb{E}\left[ \left\| u^M - U_W^{A,B} \right\|_{W^{k,p}(\Omega)} \right] \\
			&\lesssim \left( M^{-\eta} + M^{\frac{(p-1)(d+1)}{p}} N^{-\frac{1}{2}} \right) \|u\|_{\mathcal{B}^p_{k+m}(\Omega)}.
		\end{aligned}
	\end{equation*}
	Balancing the two terms by choosing $M$ optimally, i.e., $M \asymp N^{\frac{p}{2[(p-1)(d+1) + p\eta]}}$, we obtain the final error estimate \eqref{mainerror}.
\end{proof}
\begin{corollary}
	Let $\Omega$ be a bounded Lipschitz domain in $\mathbb{R}^d$ and suppose that the same assumptions as in Theorem~\ref{Theorem 1} hold. Then we have the following estimate:
	\begin{equation*}
		\mathbb{E}\left[ \left\|  u - U_W^{A,B} \right\|_{H^{k}(\Omega)} \right] \lesssim \|u\|_{H^{k+m}(\Omega)} \, N^{-\frac{\eta}{d + 2\eta + 1}},
	\end{equation*}
	where $U_W^{A,B}$ is defined as in Theorem~\ref{Theorem 1}, $\eta = \min\left( m - \frac{d}{2}, r-1 \right)$ and $m > \frac{d}{2}$.
\end{corollary}

\begin{proof}
	By the definition of the generalized Barron spectral norm (\ref{Generalized spectral norm}) and the Extension Theorem (see \cite{adams2003sobolev}), we have
	\begin{equation*}
		\begin{aligned}
			\|u\|_{\mathcal{B}_{k+m}^2(\Omega)}^2 
			&= \inf_{u_e|_\Omega = u} \biggl( \int_{\mathbb{R}^d} \left[ (1 + |\xi|)^{k+m} |\widehat{u_e}(\xi)| \right]^2 \mathrm{d}\xi \biggr) \\
			&\lesssim \inf_{u_e|_\Omega = u} \|u_e\|_{H^{k+m}(\mathbb{R}^d)}^2\lesssim \|u\|_{H^{k+m}(\Omega)}^2.
		\end{aligned}
	\end{equation*}
	The desired estimate then follows directly from Theorem~\ref{Theorem 1} with $p=2$.
\end{proof}

In the following theorem, we establish an improved convergence rate over Theorem~\ref{Theorem 1} by employing a stratified sampling strategy. Compared with Theorem~\ref{Theorem 1}, the argument requires stronger regularity of both the target function and the activation function. In particular, one additional order of regularity and decay is imposed on the activation function. The precise assumptions are stated explicitly in the theorem below.
\begin{theorem} \label{Theorem 2}
	Suppose that Assumption~\ref{assumption 1} holds and that $u\in\mathcal{B}^{1,p}_{k+m}(\Omega)$, where $p\in[2,\infty)$, $k\ge0$ is an integer, and
	$m>\frac{p-1}{p}d$. Let $\sigma:\mathbb{R}\to\mathbb{R}$ be a nonzero activation function. {Assume that, for some $r>1$,
		$
		\sigma\in W^{k+1,\infty}(\mathbb{R}),
		$
		and that there exists a constant $C_{r}>0$ such that
		$
		\left|\sigma^{(\nu)}(t)\right|
		\le
		C_{r}(1+|t|)^{-r}
		$ for every integer $0 \le \nu \le k+1$.} 
	Then, {if the parameter $M\asymp N^{\frac{p(d+3)}{2(d+1)[(p-1)(d+1)+p+p\eta]}},
		$}
	there exists a RaNN approximation $\widehat{U_W^{A,B}}$ whose parameters are sampled from the bounded set $G_\xi(M) \times G_s(M)$ and which satisfies the following estimate:
	{
		\begin{equation}\label{ieq 3.13}
			\mathbb{E}\left[ \| u - \widehat{U_W^{A,B}} \|_{W^{k,p}(\Omega)} \right] \lesssim \|u\|_{\mathcal{B}^{1, p}_{k+m}(\Omega)} \, N^{-\frac{p\eta(d+3)}{2(d+1)[(p-1)(d+1)+p+p\eta]}},
		\end{equation}
	}
	where $\eta = \min\left( m - \frac{p-1}{p} d, r-1 \right)$.
\end{theorem}

The proof of this theorem is similar to that of Theorem~\ref{Theorem 1}. The proof is provided in Appendix~\ref{appendix B} for the convenience of the reader.

\begin{remark} \label{remark 2}
	The approximation theorems established in Theorems~\ref{Theorem 1} and~\ref{Theorem 2} exhibit a complex dependence on multiple parameters, such as the decay rate $r$ of the activation function and the smoothness of the target function $u$, which is characterized by the generalized Barron spectral norm (\ref{Generalized spectral norm}). Taking $p = 2$ and letting $\eta \to \infty$ (i.e., assuming that $u$ is sufficiently smooth and the activation function $\sigma$ decays rapidly), we obtain the following asymptotic rates:
	\[
	\text{Theorem~\ref{Theorem 1}: } \mathcal{O}\left( N^{-1/2} \right) \quad \text{and} \quad \text{Theorem~\ref{Theorem 2}: } \mathcal{O}\left( N^{-1/2 - 1/(d+1)} \right).
	\]
	This result demonstrates that the obtained rates match the optimal approximation rates established in the seminal works of Barron and Xu et al.\ for SLFNs (see \cite{siegel2020approximation}). This implies that for sufficiently smooth functions, restricting the sampling set and directly approximating the truncated function $u^M$ is justified. {However, if the target function $u$ lacks sufficient smoothness, the approximation by RaNNs under uniform sampling will inevitably deteriorate and further suffer from the curse of dimensionality through the explicit dependence on $d$.}
\end{remark}

\begin{remark} \label{remark 3}
	Theorems~\ref{Theorem 1} and~\ref{Theorem 2} explain the relationship between the sampling set $G_\xi(M) \times G_s(M)$ and the target function $u$ for achieving optimal approximation when using RaNNs. Taking Theorem~\ref{Theorem 1} as an example, the optimal rate is attained when
	\[
	M \asymp N^{\frac{p}{2[(p-1)(d+1) + p\eta]}}.
	\]
	{For a fixed RaNN architecture (i.e., for a fixed number of neurons $N$), the smaller $\eta$ is, the larger $M$ must be, implying that less smooth functions require a broader parameter sampling set to achieve optimal approximation. This provides practical guidance for tailoring the sampling strategy to the inherent regularity of the target function $u$.}
\end{remark}

\section{Adaptive RaNN method} \label{section 4}
As noted in Remark~\ref{remark 3}, approximating non-smooth target functions $u$ with RaNNs requires both a sufficiently large parameter sampling range
(i.e., a sufficiently large $M$) and a sufficiently large number of neurons $N$ to achieve satisfactory accuracy. However, many functions exhibit strong local features---such as regions with sharp gradients---where globally increasing the number of uniformly sampled random basis functions $\left\lbrace \phi_i \right\rbrace_{i=1}^N$ is computationally inefficient. Inspired by adaptive finite element methods, which refine local meshes to effectively capture local solution characteristics, we develop in this section an adaptive RaNN-based algorithm for accurately solving PDEs.

\subsection{RaNNs with PoU}
To enable the RaNNs to effectively capture local features of the target function $u$, we introduce a PoU. Specifically, for a bounded domain $\Omega \subset \mathbb{R}^d$ that can be covered by the closures of a finite collection of non-overlapping open cubes $\mathcal{P}_n = \left\lbrace \omega_j \right\rbrace_{j=1}^{n}$, we define a PoU $\left\lbrace \psi_j\right\rbrace_{j=1}^n$ as follows. For any point $x \in \Omega$, let $I(x) = \left\lbrace j \big| x \in \overline{\omega_j}\right\rbrace $ be the set of indices of cubes whose closures contain $x$. Then we set
\begin{equation} 
	\psi_j(x) = \left\lbrace 
	\begin{aligned}
		&\dfrac{1}{\# I(x)}, \quad && \mathrm{if}\; j\in I(x),\\
		&0, \quad && \mathrm{otherwise},
	\end{aligned}
	\right. 
	\quad j=1,2,\dots,n.
\end{equation}
We then introduce a reference element $\omega_{\mathrm{ref}} = [-1, 1]^d$, which induces an affine transformation $\mathcal{T}_j$ mapping each physical element $\omega_j$ to $\omega_{\mathrm{ref}}$. On the reference element, we generate a family of basis functions $\left\lbrace \phi_i \right\rbrace_{i=1}^{N}$ by sampling from a probability distribution $\rho$. The resulting RaNNs with PoU can be written as
\begin{equation} \label{eq 4.2}
	\begin{split}
		\widetilde{U_W^{A,B}}(x) &= \sum_{j=1}^n \sum_{i=1}^N W_{i,j} \psi_j(x) \phi_{i}(\mathcal{T}_j(x)) \\
		&=  \sum_{j=1}^n \sum_{i=1}^N W_{i,j} \psi_j(x) \sigma \left( A_{i} \cdot \mathcal{T}_j(x) + B_{i}\right).
	\end{split}
\end{equation} 
Alternatively, one may generate independent parameter sets for each subdomain, leading to the more general expression
\begin{equation} \label{eq 4.3}
	\widetilde{U_W^{A,B}}(x) = \sum_{j=1}^n \sum_{i=1}^N W_{i,j} \psi_j(x) \sigma \left( A_{i,j} \cdot \mathcal{T}_j(x) + B_{i,j}\right),
\end{equation} 
where $\left( A_{i,j}, B_{i,j}\right)$, $i=1,\dots,N$, $j=1,\dots,n$, are uniformly sampled from $G_{\xi}(M) \times G_s(M)$ for each $i, j$. 

Given a partition $\mathcal{P}_n$ of $\Omega$, we set $\Omega_j = \omega_j \cap \Omega$, so that $\overline{\Omega} = \overline{\bigcup_{j=1}^n \Omega_j}$. The RaNNs with PoU thus define a piecewise function on $\Omega$. To later assess the continuity and regularity across subdomain interfaces, we introduce the following notation: let $\mathcal{E}^I$ denote the set of interior edges (or faces) between adjacent subdomains $\Omega_j$ within $\Omega$, and let $\mathcal{E}^D$ denote the set of boundary edges where $\partial \Omega \cap \partial \Omega_j \neq \emptyset$. For an interior edge $\Gamma \in \mathcal{E}^I$ shared by $\Omega_1$ and $\Omega_2$, the jump of a function $v$ across $\Gamma$ is defined as
\begin{equation*}
	\left[v\right]_\Gamma(x) = \lim_{y \to x, y\in \Omega_1}v(y) - \lim_{y \to x, y\in \Omega_2}v(y), \quad x \in \Gamma.
\end{equation*}

\begin{figure}[htbp] 
	\centering
	\begin{minipage}{0.9\textwidth}
		\centering
		\includegraphics[width=\linewidth]{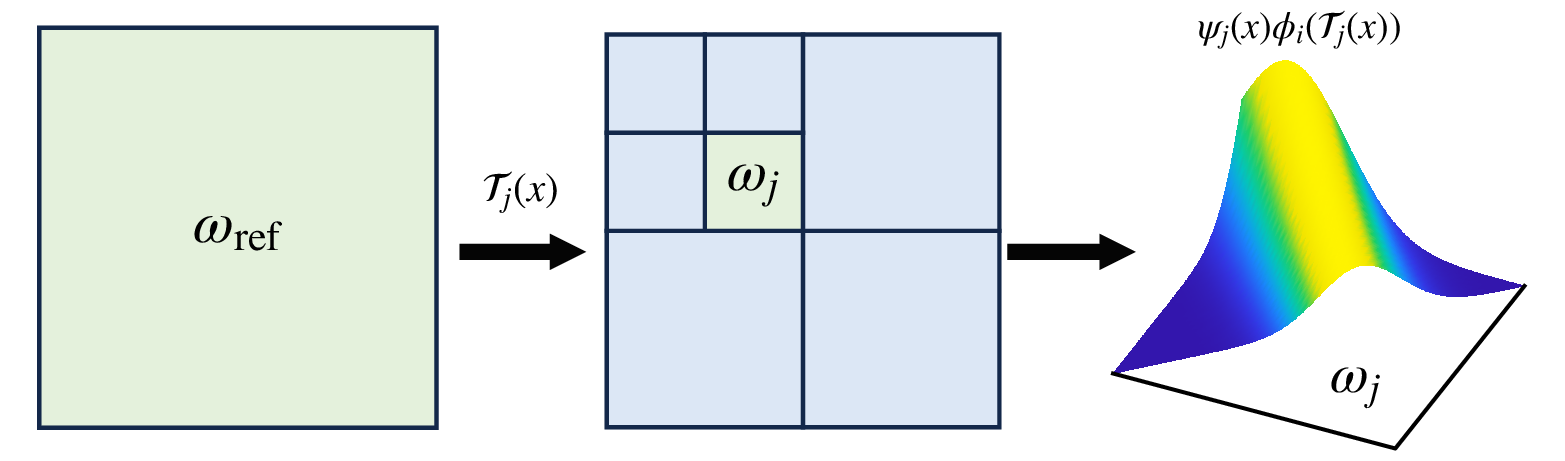}
	\end{minipage}
	\caption{Schematic diagram of two-dimensional local basis functions for the PoU-based RaNN.}\label{Fig:local_basis_function}
\end{figure}
\subsection{Adaptive PIRaNNs for solving PDEs} \label{section 4.2}
{
	For the RaNNs with a PoU (\ref{eq 4.3}), the basis functions
	\[
	\left\lbrace  \psi_j(x) \sigma \left( A_{i,j} \cdot \mathcal{T}_j(x) + B_{i,j}\right)\right\rbrace 
	\]
	are generated on the reference element $\omega_{\mathrm{ref}}$, where the parameters are drawn uniformly from the given set
	$$\left( A_{i,j}, B_{i,j}\right) \sim \mathcal{U}\left( G_{\xi}(M) \times G_s(M)\right).$$
	On the physical element $\omega_j$, the inner-layer parameters of the basis functions take the form
	\[
	A_{i,j} \cdot \mathcal{T}_j(x) + B_{i,j} = A_{i,j} \cdot \left( \mathcal{A}_j x + b_j\right) + B_{i,j},
	\]
	where $\mathcal{A}_j \in \mathbb{R}^{d \times d}$ and $b_j \in \mathbb{R}^{d}$ are associated with the affine mapping $\mathcal{T}_j$.} Consequently, (\ref{eq 4.3}) can be rewritten as follows:
\begin{equation} \label{eq 4.5}
	\widetilde{U_W^{A,B}}(x) = \sum_{j=1}^n \sum_{i=1}^N W_{i,j} \psi_j(x) \sigma \left( \widetilde{A_{i,j}} \cdot x + \widetilde{B_{i,j}}\right),
\end{equation} 
where $(\widetilde{A_{i,j}}, \widetilde{B_{i,j}}) = (A_{i,j} \cdot \mathcal{A}_j, B_{i,j} + A_{i,j} \cdot b_j)$. When the physical element $\omega_j$ is smaller than the reference element, the parameters $\bigl\{ (\widetilde{A_{i,j}}, \widetilde{B_{i,j}}) \bigr\}$ are effectively scaled and shifted relative to the original sampling set. This implies that a finer partition (i.e., one yielding smaller physical subdomains) corresponds to a locally enlarged effective sampling range for the RaNN basis functions. This mechanism perfectly aligns with the theoretical findings presented in Section~\ref{section 3}, which demonstrate that resolving features with lower regularity—such as large local gradients or singularities—inherently requires a broader sampling range. {The construction of the local basis functions is illustrated in Fig.~\ref{Fig:local_basis_function}.}

Therefore, in regions where the target function $u$ exhibits limited local smoothness, refining the PoU enables the RaNNs to resolve local features more effectively. This reformulation transforms the original challenge---namely, how to construct appropriate local basis functions under a uniform sampling distribution that must adapt to the smoothness of $u$---into the more tractable problem of designing a suitable partition of the domain $\Omega$. The latter is considerably more straightforward to implement in practice. Moreover, the construction of the PoU can be informed by a posteriori error estimates of the approximate solution, thereby providing a foundation for an adaptive computational strategy.

\begin{remark} We emphasize that although the analysis in Section~\ref{section 3} elucidates the dependence of the sampling parameter $M$ on the smoothness of the target function, it does not provide an explicit criterion for selecting $M$. To address this, here we establish a link between the effective sampling range and the mesh size of the PoU grid. In doing so, the local sampling parameter $M$ is determined adaptively as a consequence of adaptive mesh refinement. 
\end{remark}

To introduce the adaptive method, we need to assume that there exist local error indicators $\{\eta_j\}_{j=1}^n$, computable from the numerical solution on each subdomain $\Omega_j$, $j=1,\cdots,n$. The indicators provide an upper bound for the error $e = u - \widetilde{U_W^{A,B}}$ in the normed space $X$. Specifically, there exists a constant $C > 0$, depending only on the partition, such that
\begin{equation}
	\sum_{j=1}^n \| e \|_X \le C \sum_{j=1}^n \eta_j.
\end{equation}

Based on the above assumption, the standard adaptive finite element method for solving PDEs proceeds through iterations of the form
\begin{equation} \label{Adaptive workflow}
	\textsf{Solve} \rightarrow \textsf{Estimate} \rightarrow \textsf{Mark} \rightarrow \textsf{Refine}.
\end{equation}

Following this paradigm, we propose an adaptive PIRaNN algorithm, detailed below. 
The procedure begins with an initial, uniform partition of unity $\mathcal{P}_0$. For each adaptive iteration $k=0,1,2,\dots$, the following steps are performed:

\begin{itemize}
	\item \textbf{Solve:} Construct and train a PIRaNN approximation $u_k$ on the current partition $\mathcal{P}_k$ by solving the PDE.
	\item \textbf{Estimate:} Compute local error indicators $\eta_j$ for each element $\Omega_j$ derived from the numerical solution $u_k$.
	\item \textbf{Mark:} Identify a set of elements $\widehat{\mathcal{P}}_k \subseteq \mathcal{P}_k$ to be refined. In this work, we employ the Dörfler marking strategy \cite{dorfler1996convergent}: given a parameter $0 < \theta \le 1$, mark elements such that
	\[
	\eta_{\widehat{\mathcal{P}}_k} \ge \theta \, \eta_{\mathcal{P}_k},
	\]
	where $\eta_{\widehat{\mathcal{P}}_k} = \sum_{\omega_j \in \widehat{\mathcal{P}}_k} \eta_j$ and $\eta_{\mathcal{P}_k} = \sum_{j=1}^{|\mathcal{P}_k|} \eta_j$.
	\item \textbf{Refine:} Bisect (or otherwise refine) all marked elements in $\widehat{\mathcal{P}}_k$ to generate a new, finer partition $\mathcal{P}_{k+1}$.
\end{itemize}
This iterative process is repeated until a stopping criterion is met. The result is an automatically adapted PoU that dynamically aligns with the regularity of the PDE solution $u$, enabling the PIRaNN to efficiently and accurately capture its local features.


\begin{remark}
	To illustrate the adaptive PIRaNN algorithm in a concrete setting, we consider the following Poisson equation with a Dirichlet boundary condition:
	\begin{equation} \label{Poisson equation}
		\begin{aligned}
			-\Delta u &= f, \quad &&\text{in } \Omega,\\
			u &= g, \quad &&\text{on } \partial\Omega,
		\end{aligned}
	\end{equation}
	where $\Omega \subset \mathbb{R}^d$ is a bounded Lipschitz domain. For a given partition $\mathcal{P}_n$ of the domain $\Omega$, the global problem \eqref{Poisson equation} is equivalent to solving the following Poisson equation locally on each $\Omega_j$:
	\begin{equation} \label{local Poisson equation}
		\begin{aligned}
			-\Delta u_j &= f, \quad &&\text{in } \Omega_j,\\
			u_j &= g_j, \quad &&\text{on } \partial\Omega_j,
		\end{aligned}
	\end{equation}
	where the boundary data $g_j$ is given by
	\begin{equation} \label{eq 4.9}
		g_j = \left\lbrace 
		\begin{aligned}
			& g, \quad &&\text{on } \partial \Omega, \\
			& \gamma v_j, &&\text{otherwise},
		\end{aligned}
		\right. 
	\end{equation}
	where $\gamma$ is the trace operator and $v_j$ represents the solution function on the subdomains adjacent to $\Omega_j$. Let $e = u - \widetilde{U_W^{A,B}}$. Substituting $e$ into \eqref{local Poisson equation} and applying standard energy estimates, we readily obtain the following local energy estimate:
	\begin{equation}
		\| e \|_{H^1(\Omega_j)} \le C \left( \left\|  f + \Delta \widetilde{U_W^{A,B}} \right\| _{L^2(\Omega_j)}  + \left\|  g_j - \widetilde{U_W^{A,B}} \right\| _{H^{1/2}(\partial\Omega_j)} \right).
	\end{equation}
	Therefore, summing the local estimates over all elements according to the definition of $g_j$ in \eqref{eq 4.9}, we obtain the global piecewise energy estimate as follows:
	\begin{equation} \label{energy estimate}
		\begin{aligned}
			\sum_{j=1}^{n} \| e \|_{H^1(\Omega_j)} 
			\lesssim& \sum_{j=1}^{n} \left\|  f + \Delta \widetilde{U_W^{A,B}} \right\| _{L^2(\Omega_j)} 
			+ \sum_{\Gamma \in \mathcal{E}^I} \left\| \left[ \widetilde{U_W^{A,B}} \right]_\Gamma \right\|_{H^{1/2}(\Gamma)} \\
			& \qquad + \sum_{\Gamma \in \mathcal{E}^D} \left\|  \widetilde{U_W^{A,B}} - g \right\| _{H^{1/2}(\Gamma)}.
		\end{aligned}
	\end{equation}
	Inequality \eqref{energy estimate} implies that the piecewise $H^1$ error of the neural network approximation vanishes as its right-hand side tends to zero. This makes the right-hand side of \eqref{energy estimate} naturally suitable as an a posteriori error estimator. {We thus define the local error indicator $\eta_j$ as
		\begin{equation} \label{local error indicator}
			\eta_j = \left\|  f + \Delta \widetilde{U_W^{A,B}} \right\| _{L^2(\Omega_j)} + \sum_{\Gamma \in \partial \Omega_j} \left( \left\| \llbracket \widetilde{U_W^{A,B}} \rrbracket_\Gamma \right\|_{L^{2}(\Gamma)} + \beta \left\|\llbracket \nabla \widetilde{U_W^{A,B}} \rrbracket_\Gamma \right\|_{L^{2}(\Gamma)} \right),  
		\end{equation}
		where
		\begin{equation}
			\llbracket \widetilde{U_W^{A,B}} \rrbracket_\Gamma := 
			\left\lbrace 
			\begin{aligned}
				& \widetilde{U_W^{A,B}} - g, \quad && \Gamma \subset \partial \Omega,\\
				&\left[ \widetilde{U_W^{A,B}} \right]_\Gamma, \quad && \text{otherwise}.
			\end{aligned}
			\right. 
		\end{equation}
		We substitute the $H^{1/2}(\Omega_j)$ norm with the $L^2(\Gamma)$ norm of the derivatives within the local error indicator $\eta_j$, given the computational challenges associated with evaluating fractional-order Sobolev norms. Moreover, to mitigate the excessive dominance of the $\left\|\llbracket \nabla \widetilde{U_W^{A,B}} \rrbracket_\Gamma \right\|_{L^{2}(\Gamma)}$ in the a posteriori error estimate, we introduce a scaling factor $\beta = \mathrm{diam}(\Omega_j)$ as a penalty parameter for $\Gamma \in \mathcal{E}^I$. }
\end{remark}


\section{Numerical results}\label{section 5}
In this section, we present several numerical examples to validate the theoretical findings and demonstrate the performance of the adaptive PIRaNN method for solving PDEs. The linear least-squares problem arising from \eqref{eq 2.3} is solved using the \texttt{lsqminnorm} function in \texttt{MATLAB}, while the Levenberg--Marquardt algorithm is employed for nonlinear cases. {All computations are performed on a single Intel Core i7-8700 CPU (3.20 GHz).} {The errors of the numerical solutions $U_W^{A,B}$ \eqref{SLFN} or $\widetilde{U_W^{A,B}}$ \eqref{eq 4.5} obtained using the PIRaNN method, as well as the local error indicators $\eta_j$, are evaluated numerically using a $10\times10$-point Gaussian quadrature rule for two-dimensional examples, and a $5\times5\times5$-point rule for three-dimensional ones. For the Poisson equation, we employ \eqref{local error indicator} as the a posteriori error indicator.} The definition of the loss function for solving PDEs \eqref{PDEs} with \eqref{eq 4.5} is given by
\begin{equation}
	\begin{aligned}
		\mathcal{L}(W) 
		&= \left\| \mathcal{D} \left[ \widetilde{U_W^{A,B}} \right] - f \right\|^2_{L^2(\Omega)} + \lambda_{\partial \Omega} \left\| \mathcal{B} \left[ \widetilde{U_W^{A,B}} \right] - g \right\|^2_{L^2(\partial\Omega)} \\
		&\quad + \biggl( \sum_{\Gamma \in \mathcal{E}^I} \lambda_{\Gamma,1} \left\| \left[ \widetilde{U_W^{A,B}} \right]_\Gamma \right\|^2_{L^2(\Gamma)}
		+ \lambda_{\Gamma,2} \left\| \left[ \nabla \widetilde{U_W^{A,B}} \right]_\Gamma \right\|^2_{L^2(\Gamma)} \biggr),
	\end{aligned}
\end{equation}
where $\lambda_{\partial\Omega}$, $\lambda_{\Gamma,1}$ and $\lambda_{\Gamma,2}$ are penalty parameters. For the linear case, the penalty parameters can be chosen following \cite{chen2022bridging} in general. In the nonlinear case considered here, we set $\lambda_{\partial\Omega} = \lambda_{\Gamma,1} = 100$ and $\lambda_{\Gamma,2} = 10$. 

For all subsequent numerical examples, the function $\tanh(x+0.5) - \tanh(x-0.5)$ is employed as the activation function. The RaNN local basis functions $\{\phi_i\}_{i=1}^N$ in \eqref{eq 4.2} are generated on the reference element $\omega_{\mathrm{ref}}$ exactly as described in Section~\ref{section 3}. 
We set $R=1$, i.e., $G_\xi(M) = B_{M/2}^d(0)$ and $G_s(M) = [-M, M]$. To sample the parameters $\{ A_i \}_{i=1}^N$ uniformly on the ball $B_{M/2}^d(0)$, we utilize the standard Gaussian distribution in $\mathbb{R}^d$. The complete sampling procedure can be expressed as

\begin{equation}
	A_i = \dfrac{M}{2} r_i^{1/d} \frac{Y_i}{\|Y_i\|_2}, \quad i=1,\dots,N,
\end{equation}
where $Y_i$ are i.i.d. standard Gaussian samples in $\mathbb{R}^d$, and $r_i$ are i.i.d. uniform samples on $[0, 1]$. This sampling scheme ensures that each $A_i$ is distributed uniformly on the ball $B_{M/2}^d(0)$.

\begin{example} \label{example1}
	Consider the following 2D Helmholtz equation defined in the domain $\Omega = [-0.5, 0.5]^2$ with an impedance boundary condition:
	\begin{equation}
		\begin{aligned}
			-\Delta u - k^2u &= f, \quad &&\text{in } \Omega,\\
			\frac{\partial u}{\partial \mathbf{\nu}} - iku &= g, \quad &&\text{on } \partial \Omega,
		\end{aligned}
	\end{equation}
	where $i = \sqrt{-1}$ denotes the imaginary unit and $\mathbf{\nu}$ denotes the unit outward normal to $\partial \Omega$. The source term $f$ and boundary condition $g$ are chosen such that the exact solution is
	\begin{equation}
		u = \frac{\cos(kr)}{k} - \frac{\cos(k) + i\sin(k)}{k\left( J_0(k)+iJ_1(k) \right)} J_0(kr),
	\end{equation}
	where $J_a$ denotes the Bessel function of the first kind of order $a$, and $r=\sqrt{x^2+y^2}$. For every fixed $k$, the exact solution is analytic in $\Omega$. Increasing $k$ therefore increases the oscillation frequency.
\end{example}

\begin{figure}[htbp] 
	\centering
	\begin{minipage}{0.49\textwidth}
		\centering
		\includegraphics[width=\linewidth]{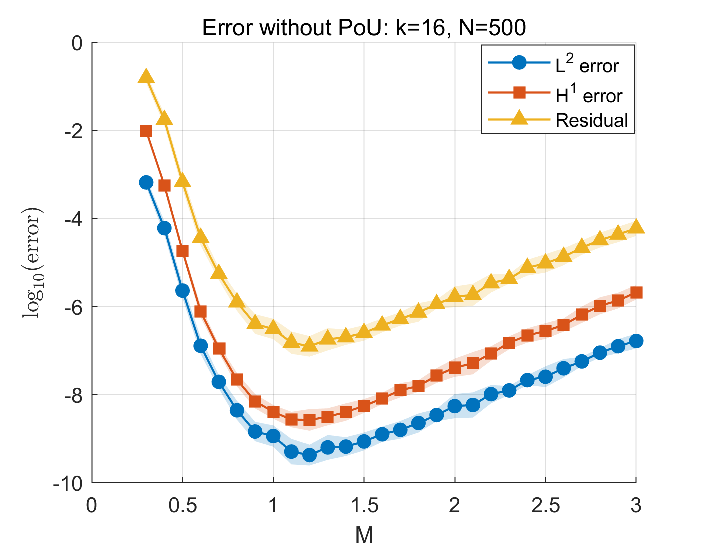}
		\\(a)
	\end{minipage}
	\hfill
	\begin{minipage}{0.49\textwidth}
		\centering
		\includegraphics[width=\linewidth]{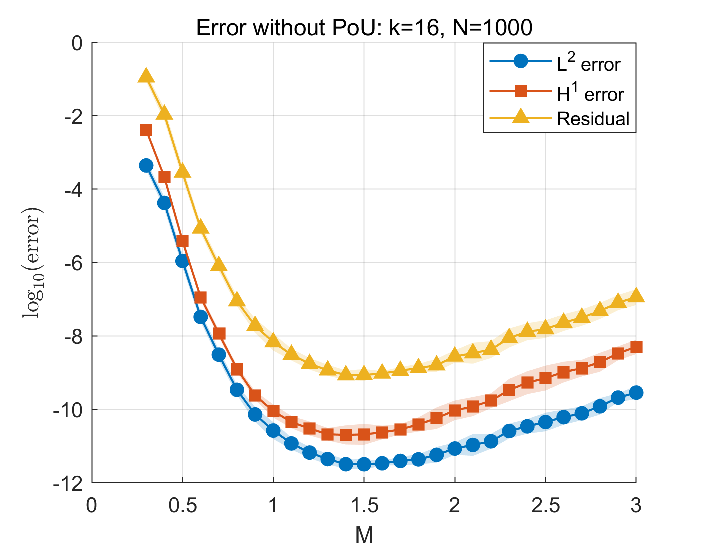}
		\\(b)
	\end{minipage}
	\vspace{0.5cm}
	\centering
	\begin{minipage}{0.49\textwidth}
		\centering
		\includegraphics[width=\linewidth]{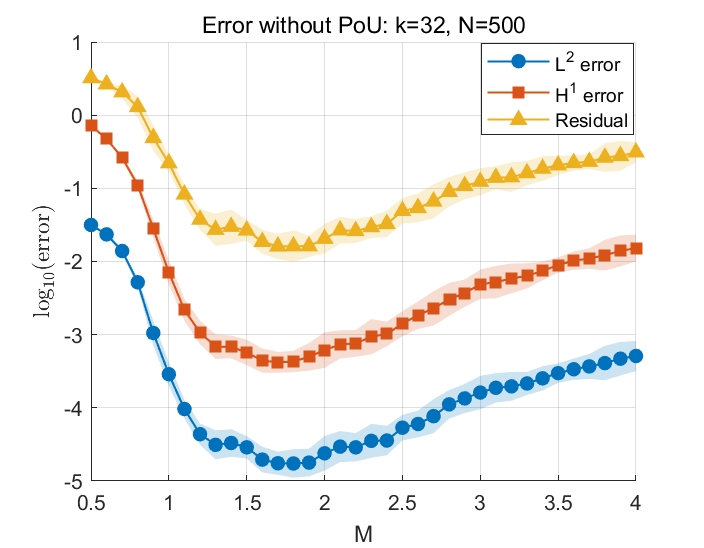}
		\\(c)
	\end{minipage}
	\hfill
	\begin{minipage}{0.49\textwidth}
		\centering
		\includegraphics[width=\linewidth]{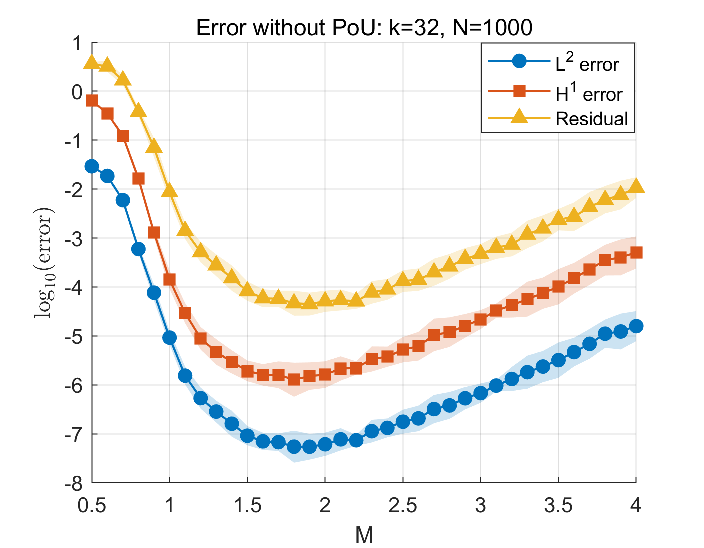}
		\\(d)
	\end{minipage}
	\caption{Error and residual curves as functions of the truncation parameter $M$ for different wavenumbers $k$ and numbers of neurons $N$, without the use of a PoU. Left column: $N = 500$; right column: $N = 1000$. Top row: $k = 16$; bottom row: $k = 32$. The curves show the means over ten independent numerical experiments, and the shaded regions indicate one standard deviation.}\label{Fig1}
\end{figure}

We first examine the influence of the truncation parameter $M$ across two wavenumbers, $k = 16$ and $k = 32$, and two network sizes, $N = 500$ and $N = 1000$. {To account for random initializations, we perform 10 independent numerical experiments using different random seeds. Figure~\ref{Fig1} displays the error and residual curves from these ten trials as functions of the truncation parameter $M$ for different wavenumbers $k$ and numbers of neurons $N$, without using a PoU.} As shown in Fig.~\ref{Fig1}, the RaNN achieves its best approximation accuracy only for an appropriately chosen \(M\); moving away from this favorable range degrades performance. The location of this range depends on both the oscillatory scale of the target function and the number of neurons. For \(k = 16\) and \(N = 500\), the best-performing value is approximately \(M=1.2\) (Fig.~\ref{Fig1}(a)), whereas increasing the wavenumber to \(k = 32\) shifts it to approximately \(M=1.7\) (Fig.~\ref{Fig1}(c)). Increasing the network size to \(N = 1000\) further shifts the favorable value upward, to approximately \(M=1.5\) for \(k = 16\) (Fig.~\ref{Fig1}(b)) and to a correspondingly larger value for \(k = 32\) (Fig.~\ref{Fig1}(d)). These results show that, within this analytic family of solutions, the effective sampling range depends jointly on the spectral content of the target function and the network size.

\begin{figure}[htbp] 
	\centering
	\begin{minipage}{0.49\textwidth}
		\centering
		\includegraphics[width=\linewidth]{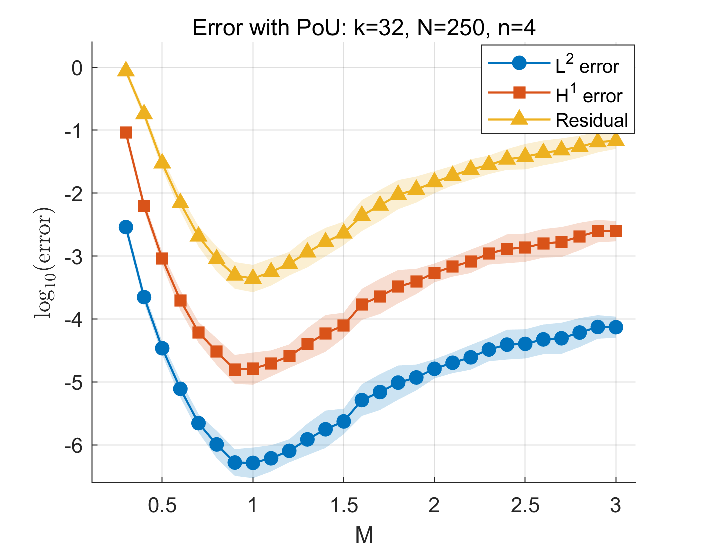}
	\end{minipage}
	\caption{Error and residual curves for different truncation parameters $M$ with a $2 \times 2$ PoU. The curves show the means over ten independent numerical experiments, and the shaded regions indicate one standard deviation.}\label{Fig2}
\end{figure}

We next assess the efficacy of the PoU augmentation strategy proposed in Section~\ref{section 4.2}. Here, RaNN basis functions are generated on a reference element \(\omega_{\text{ref}}\) with a fixed sampling parameter \(M\) and mapped affinely to each physical subdomain. For \(k = 32\), using a \(2 \times 2\) PoU in conjunction with the RaNN defined in \eqref{eq 4.3}, the favorable local truncation parameter is reduced to \(M \approx 1\) (see Fig.~\ref{Fig2}). This observation supports the affine-scaling mechanism described in Section~\ref{section 4.2}: reducing the physical element size enlarges the effective local parameter range. It therefore provides numerical motivation for controlling the local approximation scale through adaptive refinement of the PoU.

\begin{example} \label{example2}
	Consider a two-dimensional Poisson equation posed on the domain $\Omega = [0, 1]^2$. The exact solution is as follows:
	\begin{equation}
		u(x, y) = \exp{\left[ -1000 \left( (x-0.5)^2 + (y-0.5)^2\right) \right]}.
	\end{equation}
	
\end{example}

\begin{figure}[htbp]
	\centering
	\begin{minipage}{0.49\textwidth}
		\centering
		\includegraphics[width=\linewidth]{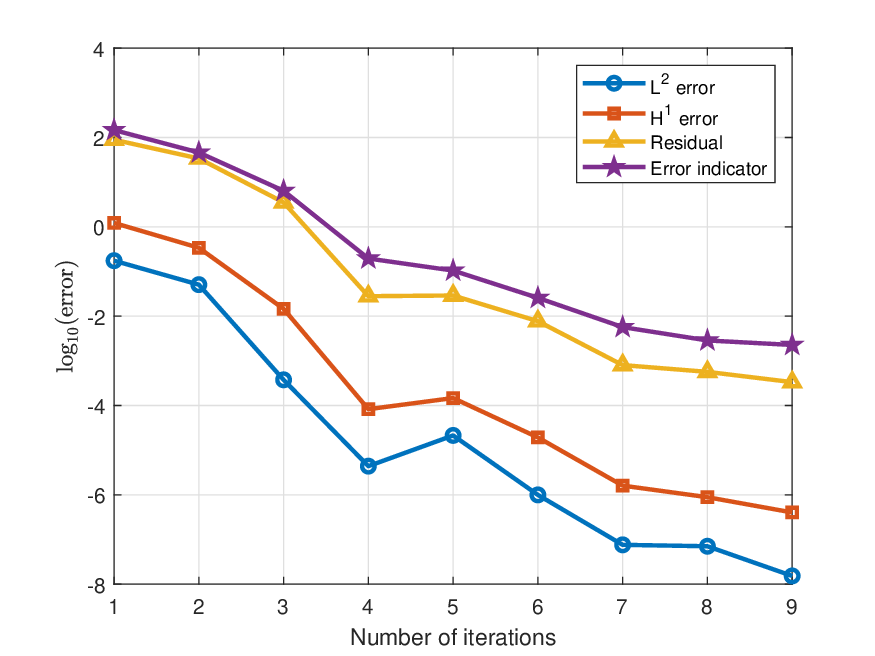}
		\\(a) 
	\end{minipage}
	\hfill
	\begin{minipage}{0.49\textwidth}
		\centering
		\includegraphics[width=\linewidth]{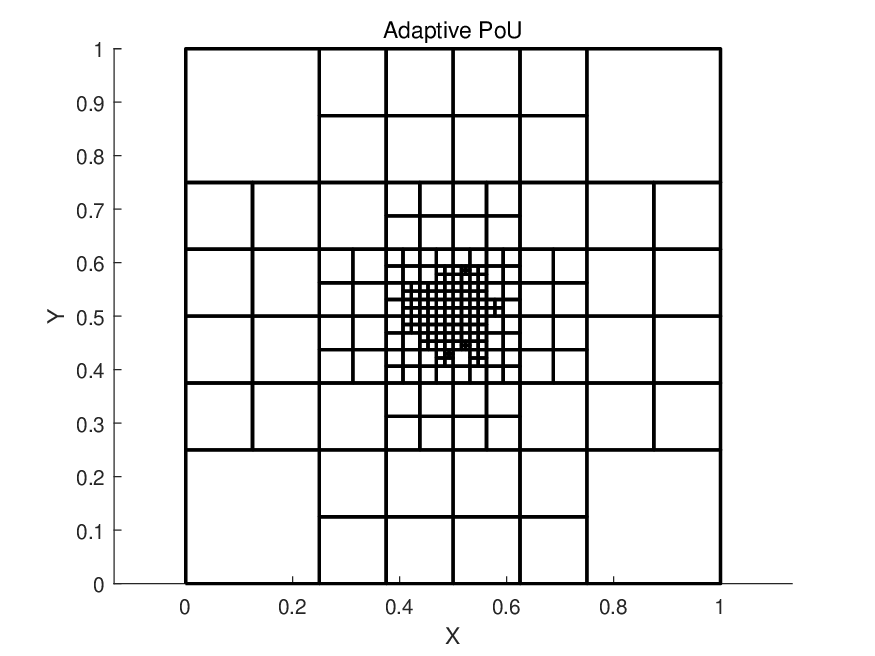}
		\\(b) 
	\end{minipage}
	
	\vspace{0.5cm}
	
	\begin{minipage}{0.49\textwidth}
		\centering
		\includegraphics[width=\linewidth]{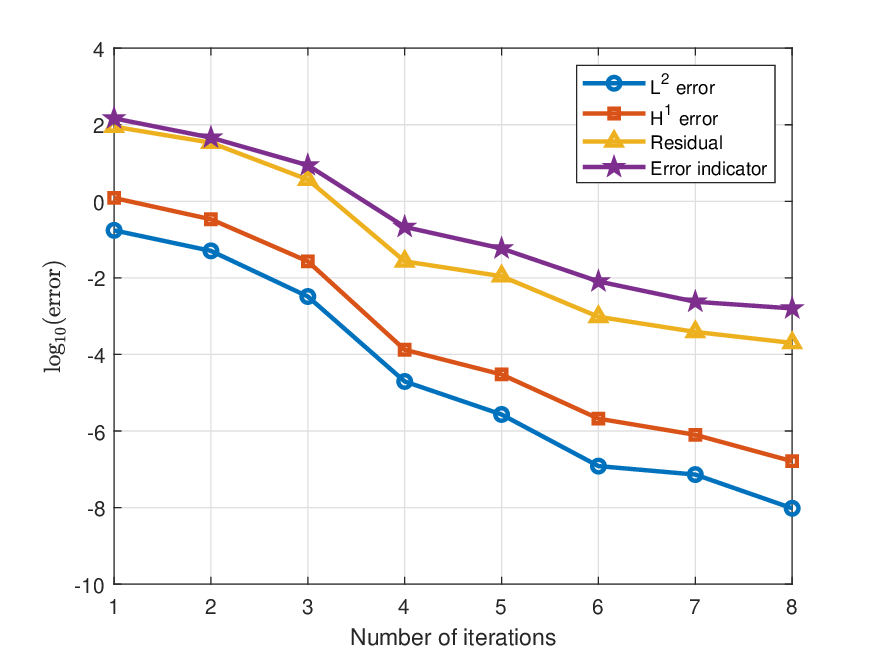}
		\\(c) 
	\end{minipage}
	\hfill
	\begin{minipage}{0.49\textwidth}
		\centering
		\includegraphics[width=\linewidth]{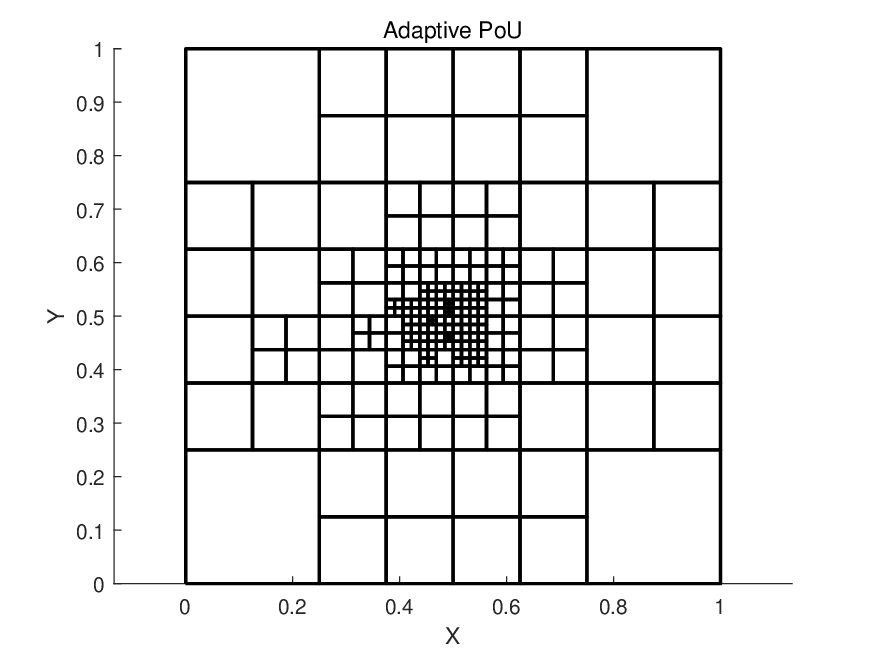}
		\\(d) 
	\end{minipage}
	
	\vspace{0.5cm}
	
	\begin{minipage}{0.49\textwidth}
		\centering
		\includegraphics[width=\linewidth]{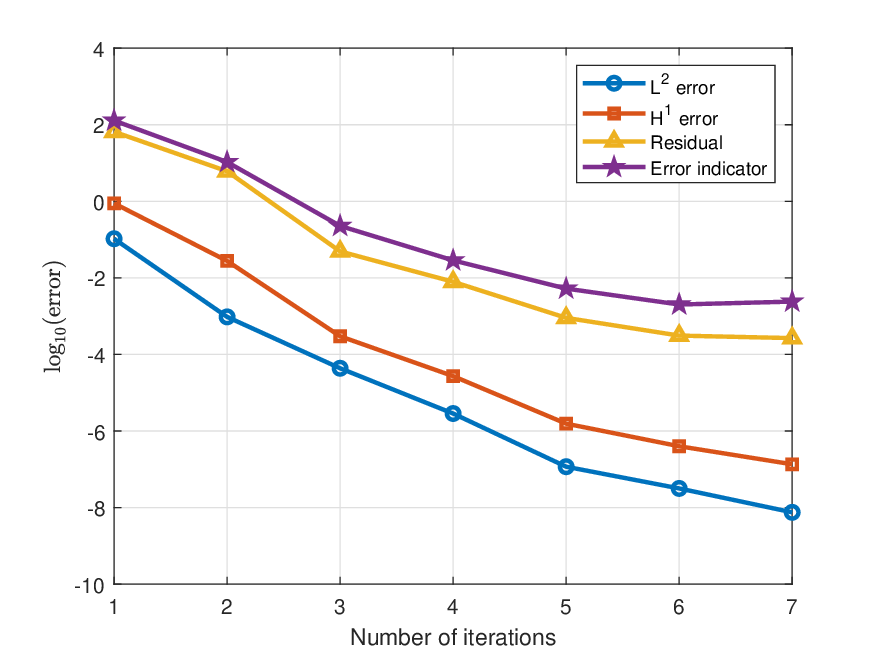}
		\\(e) 
	\end{minipage}
	\hfill
	\begin{minipage}{0.49\textwidth}
		\centering
		\includegraphics[width=\linewidth]{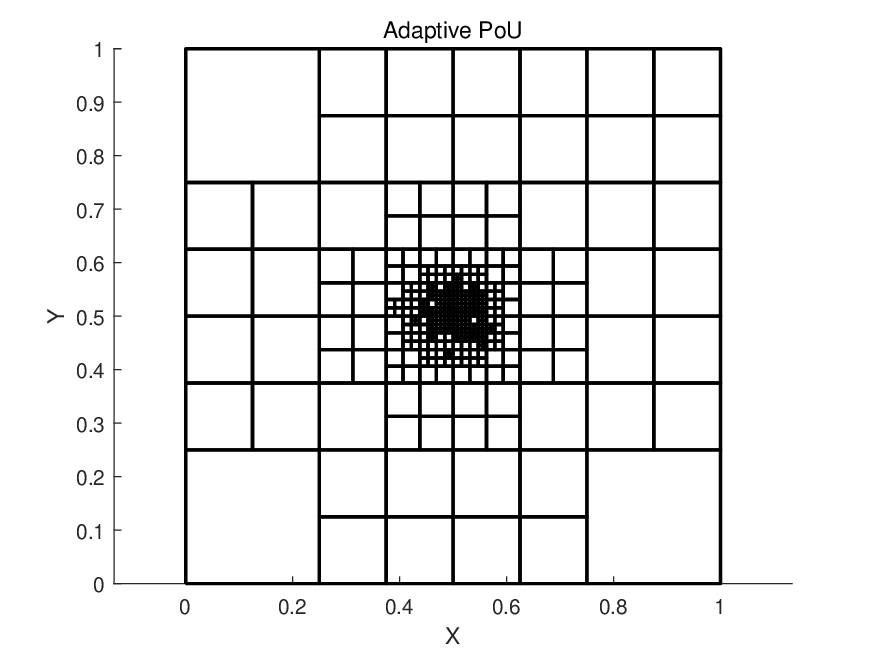}
		\\(f) 
	\end{minipage}
	
	\caption{Convergence curves and adaptive PoUs for the adaptive PIRaNNs. (a), (b): $\theta=0.6$; (c), (d): $\theta=0.7$; (e), (f): $\theta=0.8$.}\label{Fig3}
\end{figure}

In this example, we adopt the Dörfler marking strategy to drive the adaptive refinement. The Poisson equation is solved using adaptive PIRaNNs with a fixed sampling parameter $M=2$ and $N=100$ basis functions on the reference element $\omega_{\mathrm{ref}}$. The initial PoU is chosen as a uniform $4 \times 4$ grid. To assess the stability of the algorithm, we test three different marking parameters: $\theta = 0.6$, $0.7$, and $0.8$. The parameter $\theta$ primarily influences the rate at which the mesh is refined, with larger values leading to more aggressive refinement. 

The numerical results are presented in Fig.~\ref{Fig3}. Subfigures (a) and (b) correspond to $\theta = 0.6$, showing the error convergence curve and the resulting adapted PoU grid, respectively. Subfigures (c) and (d) show the results for $\theta=0.7$, whereas subfigures (e) and (f) show those for $\theta=0.8$. As expected, the adapted grids exhibit strong local refinement near the point $(0.5, 0.5)$, consistent with observations from adaptive finite element methods. This behavior aligns with our earlier analysis: regions where the solution $u$ has large local gradients can be effectively resolved by refining the PoU locally, enabling the RaNN to capture fine-scale features more accurately.

An interesting distinction between the proposed RaNN-based approach and traditional adaptive finite elements lies in the convergence behavior. In Fig.~\ref{Fig3}, we observe that the convergence rates of the $L^2(\Omega)$ and $H^1(\Omega)$ errors are nearly identical. This is consistent with the theoretical analysis in Section~\ref{section 3}. In contrast, standard finite element approximations, which rely on piecewise polynomial bases, typically exhibit a reduction in convergence order when approximating derivatives. The RaNN, employing non-polynomial representations, avoids such order reduction and maintains uniform convergence rates across different Sobolev norms.

{Because the approximation space employed by the PIRaNN method is linear, the adaptive algorithm only requires solving a linear least-squares problem at each iteration. This fundamentally circumvents the computationally expensive nonlinear optimization procedures typically required by standard PINNs. Consequently, the primary computational costs per iteration consist of assembling the global matrix $H$ and vector $T$, solving the resulting linear least-squares system, computing the local error indicators $\eta_j$, and executing the local mesh refinement alongside topological updates. In terms of memory overhead, the primary footprint arises from storing $H$ and $T$, in addition to the memory allocated by the solver during the solution of the least-squares problem~\eqref{eq 2.3}. Crucially, the integration of the PoU strategy ensures that the global coefficient matrix $H$ is highly sparse, which inherently keeps the peak memory consumption strictly manageable. For this example, the wall-clock computational times for the first six iterations, using a refinement parameter $\theta = 0.8$, are detailed in Table~\ref{tab:time_ex1} (in seconds). }

\begin{table}[htbp]
	\centering
	\small
	\caption{Example~\ref{example2}: Number of elements and computational time (in seconds) for each iteration of the adaptive PIRaNN.}
	\label{tab:time_ex1}
	\begin{tabular}{c|cccccc} 
		\toprule
		No. of Iter & No. of elements & Assembly & Solve & Indicator & Refinement & Total \\
		\midrule
		1 & 16 & 0.21 & 0.23 & 0.06 & 0.17 & 0.67 \\
		2 & 28 & 0.11 & 0.39 & 0.02 & 0.10 & 0.62 \\
		3 & 43 & 0.15 & 0.68 & 0.02 & 0.03 & 0.88 \\
		4 & 67 & 0.26 & 1.23 & 0.04 & 0.06 & 1.59 \\
		5 & 109 & 0.45 & 2.14 & 0.05 & 0.12 & 2.76 \\
		6 & 202 & 0.81 & 4.88 & 0.09 & 0.45 & 6.23 \\
		\bottomrule
	\end{tabular}
\end{table}

\begin{figure}[htbp] 
	\centering
	\begin{minipage}{0.75\textwidth}
		\centering
		\includegraphics[width=\linewidth]{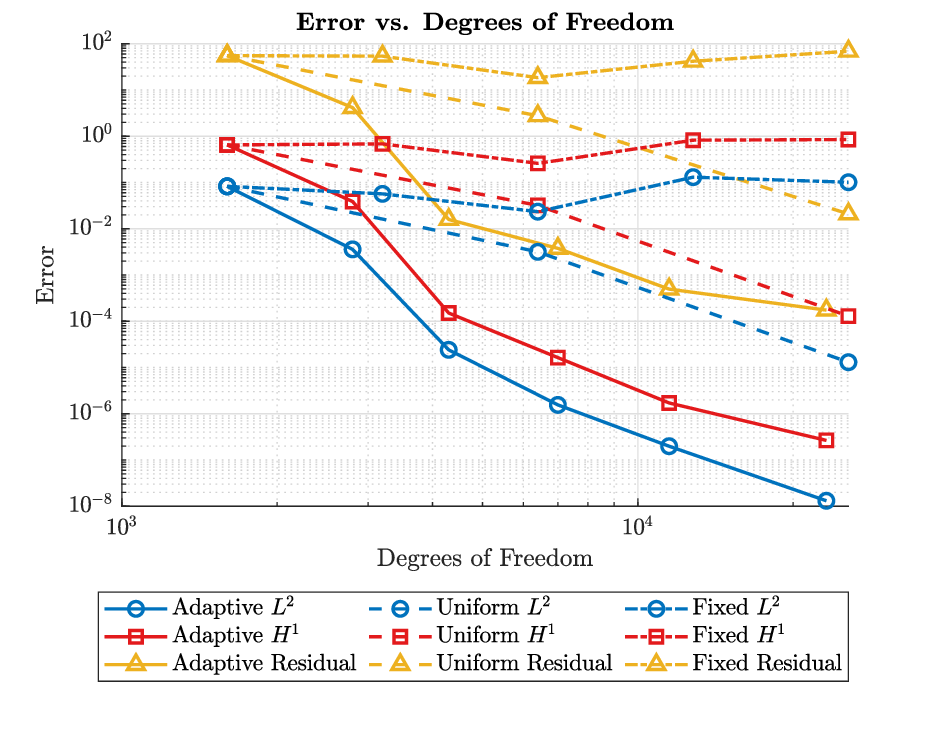}
	\end{minipage}
	\caption{Convergence histories of the errors and residuals with respect to the number of degrees of freedom (DoFs) for the Adaptive PIRaNN, Fixed PIRaNN, and Uniform PIRaNN methods.}
	\label{fig:comparison_curve}
\end{figure}

{To demonstrate the computational efficiency of the adaptive PIRaNN method in resolving problems with steep local gradients or localized singularities, we compare it against two non-adaptive baselines:
	\begin{itemize}
		\item Fixed PIRaNN: Employs a fixed $4\times 4$ PoU mesh while progressively increasing the number of local basis functions per element.
		\item Uniform PIRaNN: Applies uniform mesh refinement across the entire computational domain while keeping the number of basis functions per element constant.
	\end{itemize}
	The decay curves of the errors and residuals with respect to the number of degrees of freedom (DoFs) for all three methods are illustrated in Fig.~\ref{fig:comparison_curve}.
	As can be observed, the adaptive PIRaNN method yields the highest accuracy for the same number of DoFs, which implies that it can effectively reduce computational costs to reach a target accuracy. Notably, for the method with the fixed mesh, the error increases instead of decreasing. This is because an excessive number of local basis functions with poor orthogonality leads to a rapid growth in the condition number of the matrix $H$, ultimately inducing numerical instability. Furthermore, the sampling range of the hidden-layer parameters $\left\lbrace (A_i,B_i)\right\rbrace_{i=1}^N$ is too narrow, causing the basis functions to fail to capture features with large local gradients. This phenomenon is consistent with the theoretical analysis in Section~\ref{section 3} and aligns with Example~\ref{example1}.  
}

\begin{example} \label{example3}
	Consider a three-dimensional Poisson equation posed on $\Omega = [0, 0.5]^3$ whose exact solution is
	\begin{equation}
		u(x,y,z) = \exp{\left[ -1000 \left( (x-0.5)^2 + (y-0.5)^2 + (z-0.5)^2\right) \right]}.
	\end{equation}
\end{example}

In this test, we again employ the Dörfler marking strategy with a refinement parameter $\theta = 0.7$. The number of basis functions generated on the reference element is fixed at $N=200$, and the initial PoU is taken as a uniform $2 \times 2 \times 2$ grid. The numerical results are presented in Fig.~\ref{Fig4}. As shown, the PoU grid becomes highly refined near the point $(0.5, 0.5, 0.5)$. Notably, the observed convergence rate is lower than that observed in the two-dimensional example (Fig.~\ref{Fig3}), suggesting that the performance of uniformly sampled PIRaNNs may be influenced by the dimension $d$. This observation is consistent with the theoretical expectation that higher-dimensional problems pose challenges for approximation, even within the RaNN framework. As established in Theorems~\ref{Theorem 1} and~\ref{Theorem 2}, the exponent in the convergence rate with respect to the number of neurons $N$ scales as $\mathcal{O}(1 / d)$ with respect to the spatial dimension $d$. However, for functions with sufficient smoothness, the adverse effect of dimensionality on the convergence rate is significantly alleviated, as shown in Remark~\ref{remark 2}.

\begin{figure}[htbp] 
	\centering
	\begin{minipage}{0.49\textwidth}
		\centering
		\includegraphics[width=\linewidth]{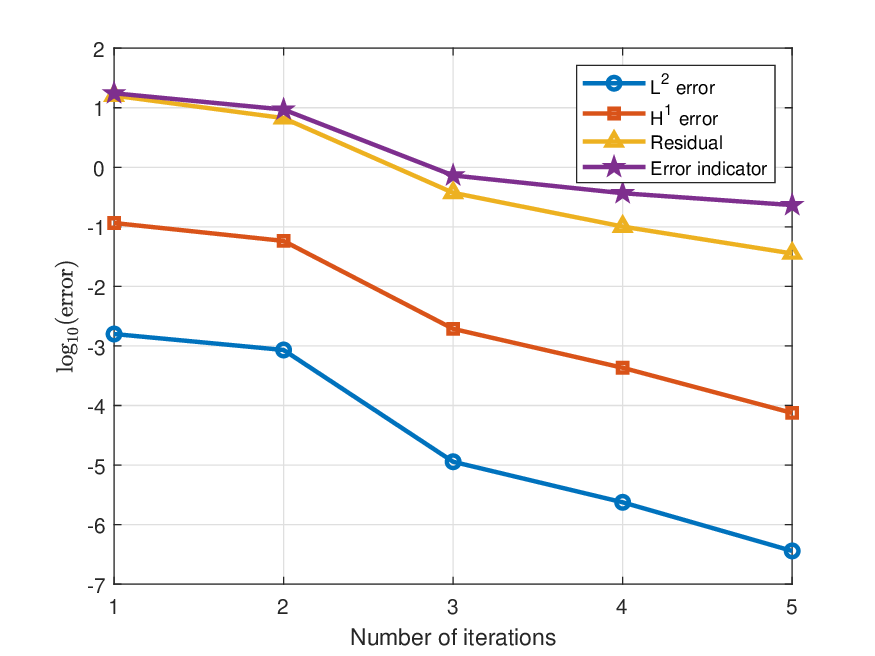}
		\\(a) 
	\end{minipage}
	\hfill
	\begin{minipage}{0.49\textwidth}
		\centering
		\includegraphics[width=\linewidth]{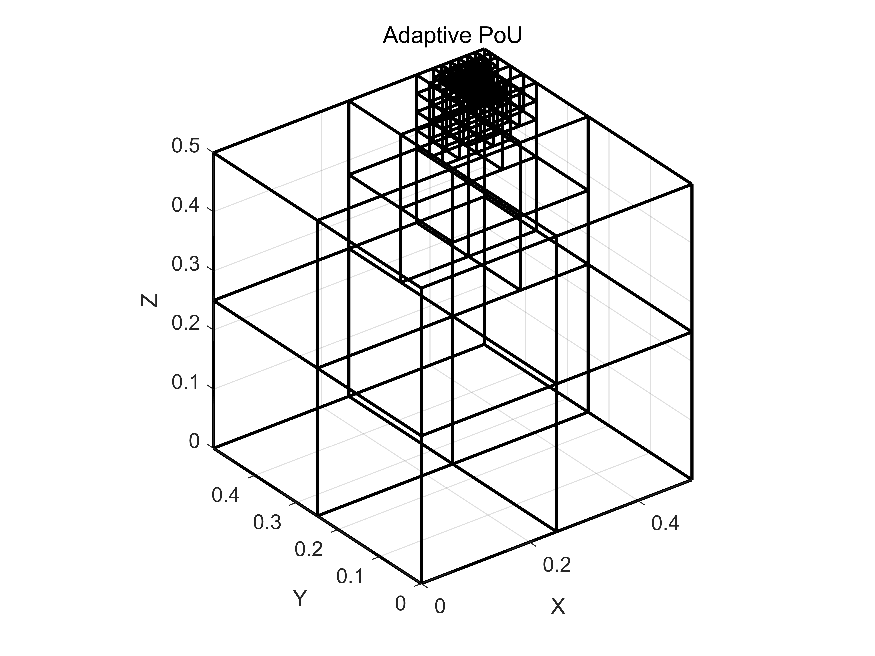}
		\\(b) 
	\end{minipage}
	\caption{Convergence curve and adaptive PoU of Example~\ref{example3} obtained using the adaptive PIRaNN method.}\label{Fig4}
\end{figure}

\begin{example} \label{example4}
	Consider the Poisson equation posed on the L-shaped domain $\Omega$ depicted in Fig.~\ref{Fig5}, with Dirichlet boundary conditions chosen such that the exact solution is given by
	\begin{equation}
		u = r^{2/3} \sin(2\theta / 3).
	\end{equation}
	The primary challenge of this problem stems from the nonconvex corner at the origin, where the solution exhibits singular behavior due to reduced local regularity.
\end{example}

\begin{figure}[htbp]
	\centering
	\begin{minipage}{0.49\textwidth}
		\centering
		\includegraphics[width=\linewidth]{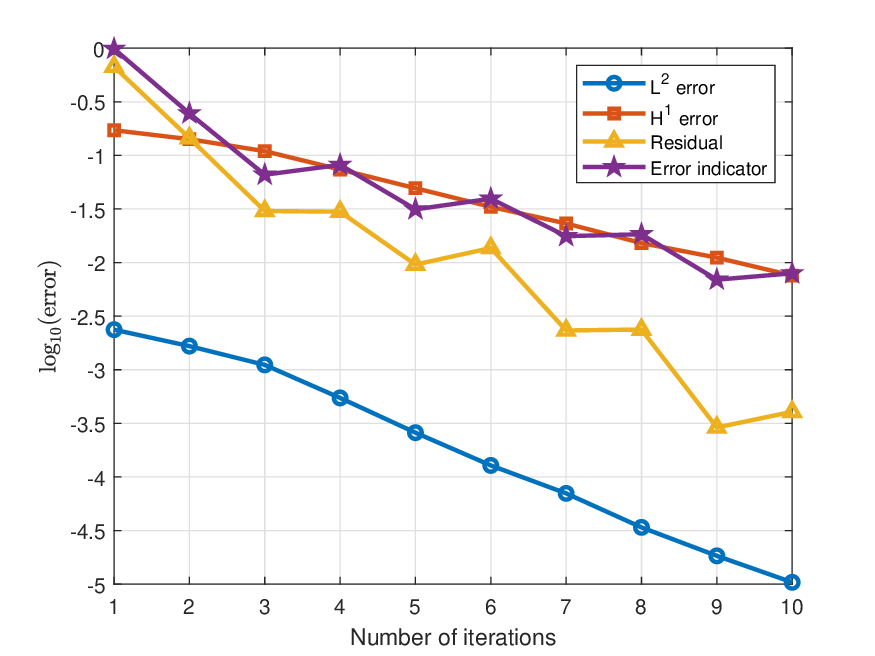}
		\\(a) 
	\end{minipage}
	\hfill
	\begin{minipage}{0.49\textwidth}
		\centering
		\includegraphics[width=\linewidth]{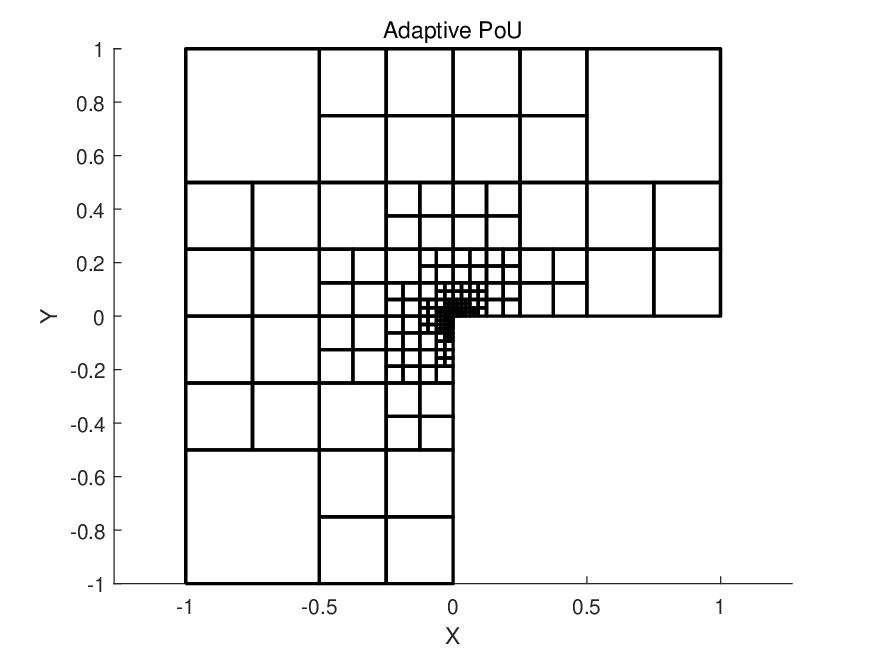}
		\\(b) 
	\end{minipage}
	\caption{Convergence curve and adaptive PoU of Example~\ref{example4} obtained using the adaptive PIRaNN method.}\label{Fig5}
\end{figure}

In this test, we set $\theta = 0.7$, $\lambda_{\partial\Omega} = \lambda_{\Gamma,1} = 100$, $\lambda_{\Gamma,2} = 10$, and the RaNN parameters are the same as in Example~\ref{example2}. The left panel of Fig.~\ref{Fig5} displays the convergence curves of the error, residual, and error indicator, while the right panel shows the adapted PoU. The adaptive mesh is refined near the origin to better capture the local solution behavior, owing to the singularity that degrades the regularity of $u$ in that region. As predicted by the theory, this local refinement effectively reduces the approximation error. However, precisely because of the limited regularity of $u$, the convergence rate of the error in the $H^1(\Omega)$ norm is lower than that in the $L^2(\Omega)$ norm, and consequently does not achieve the same rate as that observed in Example~\ref{example2}. This behavior is consistent with the theoretical results in Section~\ref{section 3}, which indicate that convergence rates in stronger norms are more sensitive to the smoothness of the target function.

\begin{example} \label{example5}
	Consider the following one-dimensional viscous Burgers' equation
	\begin{equation}
		\begin{aligned}
			u_t + u u_x &= \nu u_{xx}, \quad &&\text{in } (0, T) \times D,\\
			u &= 0, \quad &&\text{on } (0, T) \times \partial D, \\
			u(0, x) &= \sin(2\pi x), \quad &&\text{in } D,
		\end{aligned}
	\end{equation}
	where the viscous coefficient $\nu = 0.01 / \pi$, the spatial domain $D = (0, 1)$ and the final time is $T = 1$.
\end{example}

For this nonlinear PDE, we employ a space-time formulation of PIRaNNs, treating the time variable as an additional coordinate on an equal footing with the spatial variables. Hence, the computational domain is $\Omega = (0, T) \times D$ and the RaNN parameters are the same as in Example~\ref{example2}. It has been shown in \cite{zhu2025two} that the $L^2$ error for the Burgers' equation can be bounded by the residuals of the PDE, initial condition, and boundary conditions. Based on this theoretical result, we define a local error indicator $\eta_j$ on each subdomain $\Omega_j$ (with $j=1,\dots,n$) for the RaNN approximation with a PoU:
{
	\begin{equation}
		\begin{aligned}
			\eta_j 
			&= \left\| \left( \widetilde{U_W^{A,B}} \right)_t + \widetilde{U_W^{A,B}} \left( \widetilde{U_W^{A,B}} \right)_x - \nu \left( \widetilde{U_W^{A,B}} \right)_{xx} \right\|_{L^2(\Omega_j)} \\
			&\quad + \sum_{\Gamma \in \partial \Omega_j} \left( \left\| \llbracket \widetilde{U_W^{A,B}} \rrbracket_\Gamma \right\|_{L^{2}(\Gamma)} + \beta_\Gamma \left\| \llbracket \nabla \widetilde{U_W^{A,B}} \rrbracket_\Gamma \right\|_{L^{2}(\Gamma)} \right),
		\end{aligned}
	\end{equation}
	where the jump term is defined according to the type of interface $\Gamma$:
	\begin{equation}
		\llbracket \widetilde{U_W^{A,B}} \rrbracket_\Gamma := 
		\left\lbrace 
		\begin{aligned}
			& \widetilde{U_W^{A,B}}(0, x) - \sin(2\pi x), \quad && \Gamma \subset D \text{ (initial condition)},\\
			& \widetilde{U_W^{A,B}}(t, x), \quad && \Gamma \subset [0, T] \times \partial D \text{ (boundary condition)}, \\
			&\left[ \widetilde{U_W^{A,B}} \right]_\Gamma, \quad && \Gamma \in \mathcal{E}^I \text{ (interior interface)}.
		\end{aligned}
		\right. 
	\end{equation}
	The penalty parameter $\beta_\Gamma$ is set as
	\begin{equation}
		\beta_\Gamma = 
		\left\lbrace 
		\begin{aligned}
			& \mathrm{diam}(\Omega_j), \quad && \Gamma \in \mathcal{E}^I,\\
			&0, \quad && \text{otherwise}.
		\end{aligned}
		\right. 
	\end{equation}
}
To generate a high-fidelity reference solution for the Burgers' equation, we employ the Chebfun package \cite{driscoll2014chebfun} with a Fourier spectral discretization in space and a fourth-order exponential time-differencing scheme \cite{cox2002exponential} with a time step of $10^{-4}$. Temporal snapshots of the solution are saved every $\Delta t = 0.01$, yielding a total of 101 snapshots. The reference solution is evaluated on a $201 \times 101$ spatio-temporal grid, providing a sufficiently accurate baseline for error assessment.

The numerical results are illustrated in Figures~\ref{Fig6} and~\ref{Fig7}. The adaptive parameter is set to $\theta = 0.6$. As shown in Fig.~\ref{Fig6}(a), both the $L^2$ error and the residual of the numerical solution decrease steadily throughout the adaptive iterations, confirming the effectiveness of the adaptive refinement strategy. The final adapted mesh obtained from this process is presented in Fig.~\ref{Fig6}(b). It is evident that the mesh is refined preferentially in the vicinity of the developing shock, enabling the RaNN to capture the sharp front with improved accuracy. Fig.~\ref{Fig7} displays the numerical solution and the corresponding absolute error at selected iterations. Initially, the coarse mesh is insufficient for the RaNN to resolve the shock, leading to significant errors in the shock region. As the adaptive process proceeds and the mesh is locally refined, the shock is captured with increasing precision, and the overall error is substantially reduced. These results demonstrate the capability of the adaptive PIRaNN framework to automatically detect and resolve localized solution features in nonlinear, time-dependent problems.

\begin{figure}[htbp] 
	\centering
	\begin{minipage}{0.49\textwidth}
		\centering
		\includegraphics[width=\linewidth]{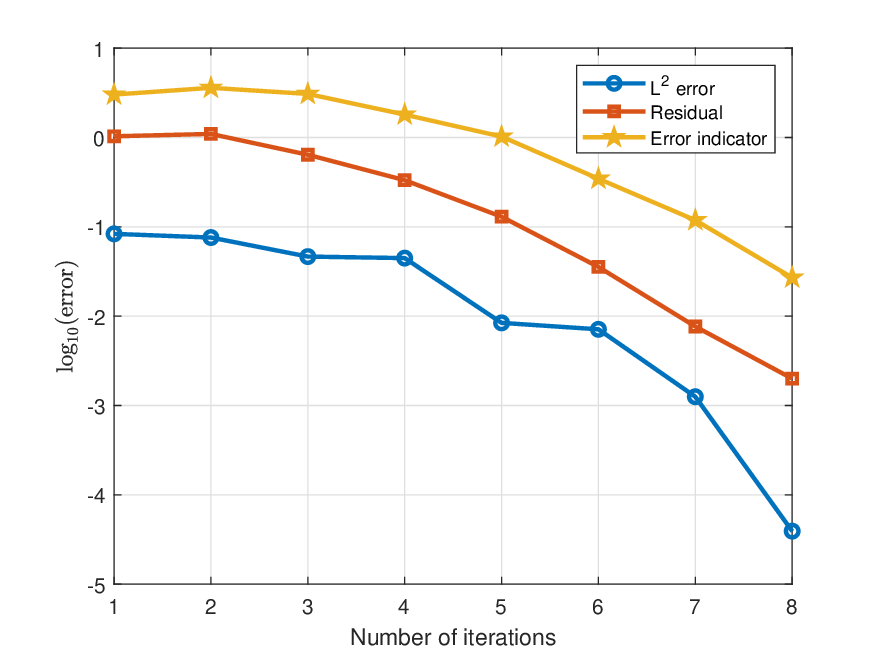}
		\\(a) 
	\end{minipage}
	\hfill
	\begin{minipage}{0.49\textwidth}
		\centering
		\includegraphics[width=\linewidth]{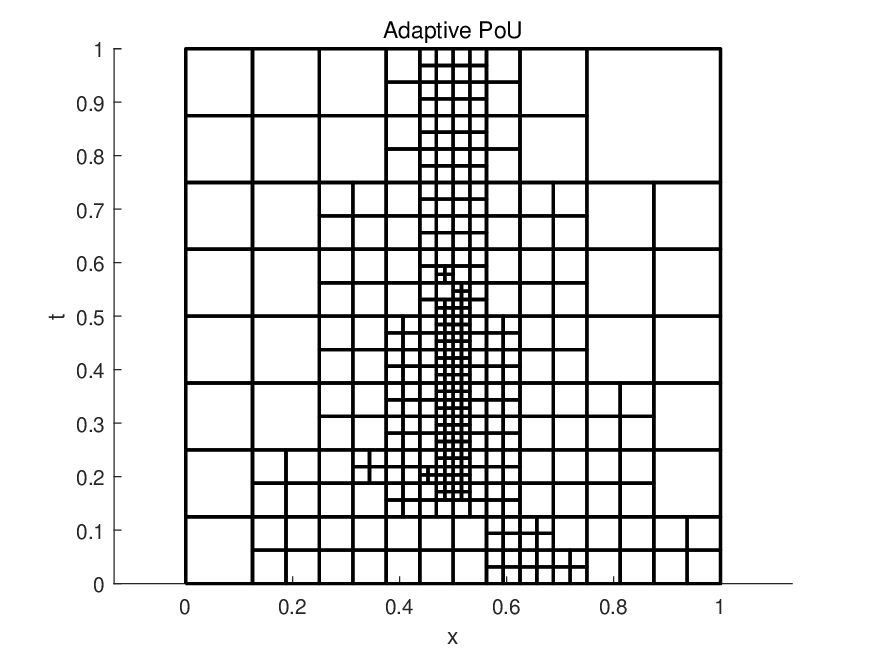}
		\\(b) 
	\end{minipage}
	\caption{Convergence curve and adaptive PoU of Example~\ref{example5} obtained using the adaptive PIRaNN method.}\label{Fig6}
\end{figure}

\begin{figure}[htbp] 
	\centering
	\begin{minipage}{0.12\textwidth}
		\quad
	\end{minipage}
	\centering
	\begin{minipage}{0.28\textwidth}
		\centering
		\small{Step 1.}
	\end{minipage}
	\hfill
	\begin{minipage}{0.28\textwidth}
		\centering
		\small{Step 5.}
	\end{minipage}
	\hfill
	\begin{minipage}{0.28\textwidth}
		\centering
		\small{Step 8.}
	\end{minipage}
	
	\vspace{0.1 cm}
	\centering
	\begin{minipage}{0.12\textwidth}
		\centering
		\small{Numerical solution}
	\end{minipage}
	\centering
	\begin{minipage}{0.28\textwidth}
		\centering
		\includegraphics[width=\linewidth]{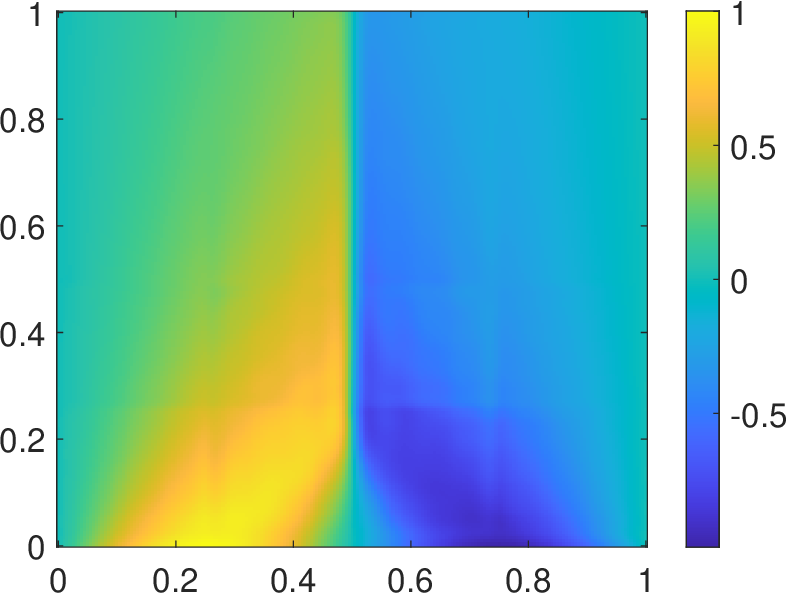}
	\end{minipage}
	\hfill
	\begin{minipage}{0.28\textwidth}
		\centering
		\includegraphics[width=\linewidth]{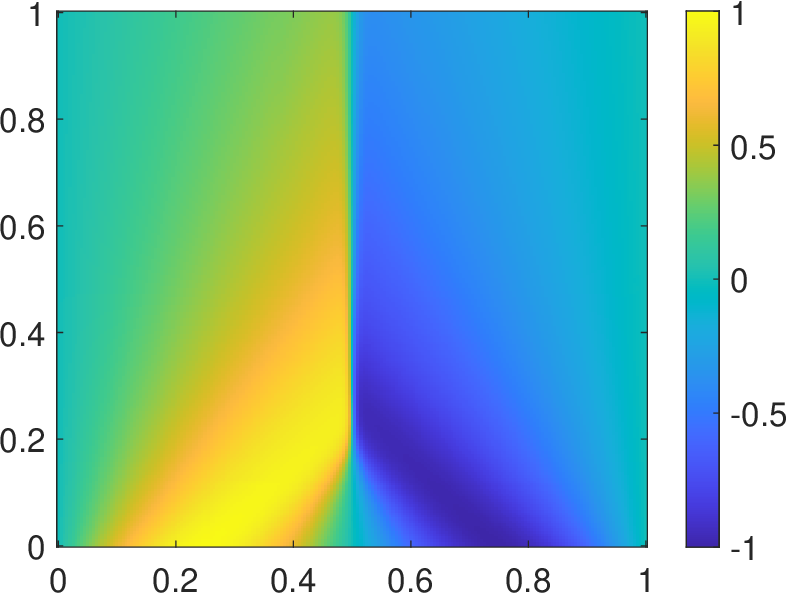}
	\end{minipage}
	\hfill
	\begin{minipage}{0.28\textwidth}
		\centering
		\includegraphics[width=\linewidth]{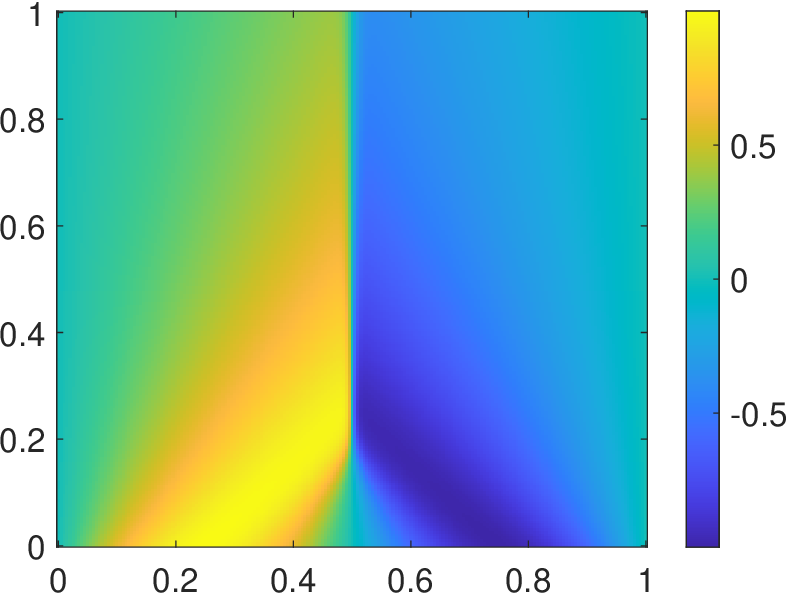}
	\end{minipage}
	
	\vspace{0.5 cm}
	
	\centering
	\begin{minipage}{0.12\textwidth}
		\centering
		\small{Absolute error}
	\end{minipage}
	\centering
	\begin{minipage}{0.28\textwidth}
		\centering
		\includegraphics[width=\linewidth]{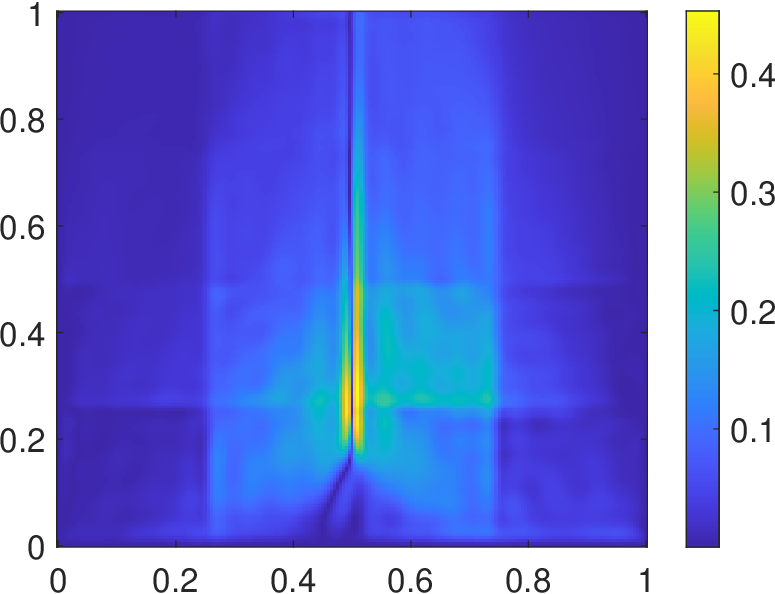}
	\end{minipage}
	\hfill
	\begin{minipage}{0.28\textwidth}
		\centering
		\includegraphics[width=\linewidth]{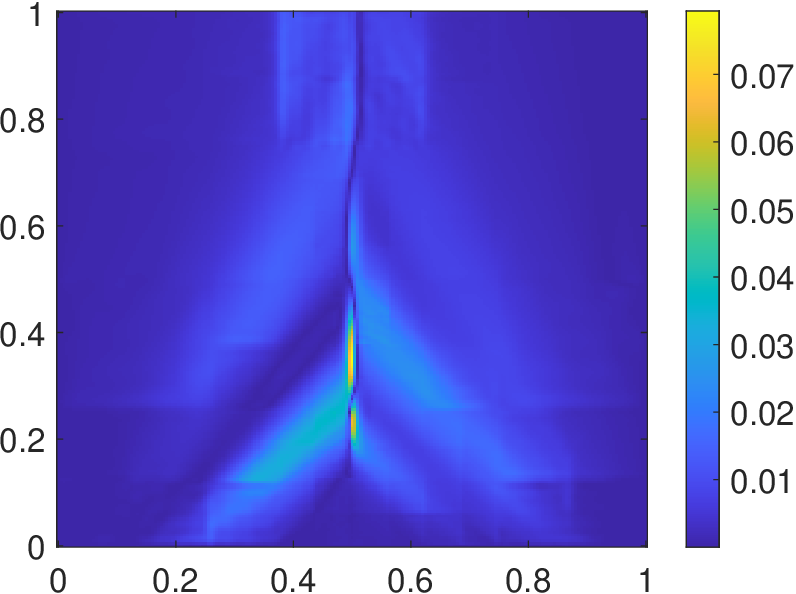}
	\end{minipage}
	\hfill
	\begin{minipage}{0.28\textwidth}
		\centering
		\includegraphics[width=\linewidth]{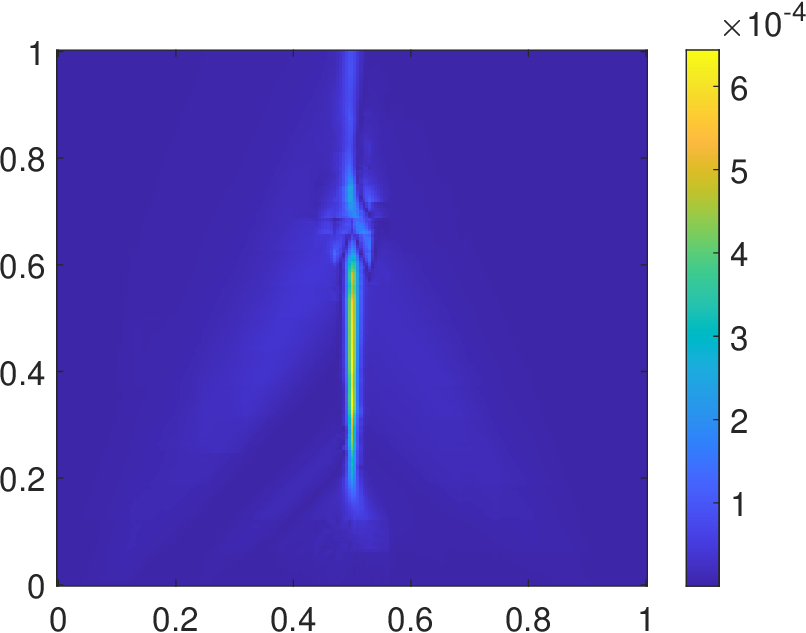}
	\end{minipage}
	\caption{Numerical results obtained using the adaptive PIRaNN method for Example~\ref{example5}. The top row shows the numerical solution $\widetilde{U_W^{A,B}}$, and the bottom row shows the absolute error between the numerical and reference solutions.}\label{Fig7}
\end{figure}

\section{Conclusion} \label{section 6}
In this paper, we have developed a comprehensive theoretical and algorithmic framework for solving PDEs using RaNNs. We introduced a generalized Barron spectral space $\mathcal{B}_s^{k,p}(\Omega)$ to characterize functions that can be efficiently approximated by RaNNs whose hidden-layer parameters are uniformly sampled from a bounded set. For functions in this space, we established explicit convergence rates in Sobolev norms, revealing a fundamental relationship between the required sampling range and the smoothness of the target function. Specifically, we showed that less smooth functions necessitate a larger sampling range to achieve optimal approximation, a result that provides theoretical guidance for parameter selection in practice.

Motivated by this insight, we integrated a PoU with RaNNs to develop an adaptive PIRaNN method. By linking the parameter sampling range to the local element size via affine mappings, the PoU framework effectively translates the challenge of generating appropriate basis functions into the more tractable task of designing a solution-adaptive partition. We further incorporated a posteriori error estimates and the Dörfler marking strategy to drive adaptive refinement, enabling the network to automatically concentrate computational resources in regions where the solution exhibits limited regularity.

A series of numerical experiments validated both the theoretical analysis and the practical effectiveness of the proposed approach. The results confirmed that the convergence rate of RaNNs depends critically on the sampling range relative to the smoothness of the solution, and that the adaptive PIRaNN method successfully captures localized features such as singularities and shocks. Extensions to the viscous Burgers' equation further demonstrated the applicability of the framework to nonlinear, time-dependent problems.

There are several directions for future work. {Currently, our research is largely confined to the uniform sampling of parameters; future studies should explore alternative sampling distributions and more complex sampling strategies.} {This includes establishing the corresponding approximation theorems and examining whether they can effectively improve the approximation rate of RaNN.} Additionally, the theoretical analysis could be broadened to encompass deeper architectures and more general activation functions. {Another key direction is to refine the adaptive algorithmic strategies and establish rigorous convergence theorems for them. Ultimately, we intend to tackle high-dimensional and more complex PDEs.}

\bibliographystyle{siamplain}
\bibliography{references}

@article{xu2020finite,
	title={The finite neuron method and convergence analysis},
	author={Xu, Jinchao},
	journal={Communications in Computational Physics},
	volume={28},
	number={5},
	pages={1707--1745},
	year={2020},
	publisher={Global Science Press}
}

@article{siegel2020approximation,
	title={Approximation rates for neural networks with general activation functions},
	author={Siegel, Jonathan W and Xu, Jinchao},
	journal={Neural Networks},
	volume={128},
	pages={313--321},
	year={2020},
	publisher={Elsevier}
}

@article{huang2006extreme,
	title={Extreme learning machine: theory and applications},
	author={Huang, Guang-Bin and Zhu, Qin-Yu and Siew, Chee-Kheong},
	journal={Neurocomputing},
	volume={70},
	number={1-3},
	pages={489--501},
	year={2006},
	publisher={Elsevier}
}

@article{raissi2019physics,
	title={Physics-informed neural networks: A deep learning framework for solving forward and inverse problems involving nonlinear partial differential equations},
	author={Raissi, Maziar and Perdikaris, Paris and Karniadakis, George E},
	journal={Journal of Computational physics},
	volume={378},
	pages={686--707},
	year={2019},
	publisher={Elsevier}
}

@article{chen2022bridging,
	title={Bridging traditional and machine learning-based algorithms for solving PDEs: the random feature method},
	author={Chen, Jingrun and Chi, Xurong and E, Weinan and Yang, Zhouwang},
	journal={Journal of Machine Learning},
	volume={1},
	number={3},
	pages={268--298},
	year={2022}
}

@article{dong2021local,
	title={Local extreme learning machines and domain decomposition for solving linear and nonlinear partial differential equations},
	author={Dong, Suchuan and Li, Zongwei},
	journal={Computer Methods in Applied Mechanics and Engineering},
	volume={387},
	pages={114129},
	year={2021},
	publisher={Elsevier}
}

@article{barron2002universal,
	title={Universal approximation bounds for superpositions of a sigmoidal function},
	author={Barron, Andrew R},
	journal={IEEE Transactions on Information theory},
	volume={39},
	number={3},
	pages={930--945},
	year={2002},
	publisher={IEEE}
}

@article{klusowski2016risk,
	title={Risk bounds for high-dimensional ridge function combinations including neural networks},
	author={Klusowski, Jason M and Barron, Andrew R},
	journal={arXiv preprint arXiv:1607.01434},
	year={2016}
}

@article{siegel2022high,
	title={High-order approximation rates for shallow neural networks with cosine and ReLUk activation functions},
	author={Siegel, Jonathan W and Xu, Jinchao},
	journal={Applied and Computational Harmonic Analysis},
	volume={58},
	pages={1--26},
	year={2022},
	publisher={Elsevier}
}

@book{folland1999real,
	title={Real analysis: modern techniques and their applications},
	author={Folland, Gerald B},
	year={1999},
	publisher={John Wiley \& Sons}
}

@book{adams2003sobolev,
	title={Sobolev spaces},
	author={Adams, Robert A and Fournier, John JF},
	volume={140},
	year={2003},
	publisher={Elsevier}
}

@book{evans2022partial,
	title={Partial differential equations},
	author={Evans, Lawrence C},
	volume={19},
	year={2022},
	publisher={American Mathematical Society}
}

@article{lu2021learning,
	title={Learning nonlinear operators via DeepONet based on the universal approximation theorem of operators},
	author={Lu, Lu and Jin, Pengzhan and Pang, Guofei and Zhang, Zhongqiang and Karniadakis, George Em},
	journal={Nature Machine Intelligence},
	volume={3},
	number={3},
	pages={218--229},
	year={2021},
	publisher={Nature Publishing Group UK London}
}

@inproceedings{li2020fourier,
	title={Fourier neural operator for parametric partial differential equations},
	author={Li, Zongyi and Kovachki, Nikola and Azizzadenesheli, Kamyar and Liu, Burigede and Bhattacharya, Kaushik and Stuart, Andrew and Anandkumar, Anima},
	booktitle={International Conference on Learning Representations},
	year={2021}
}

@article{guhring2021approximation,
	title={Approximation rates for neural networks with encodable weights in smoothness spaces},
	author={G{\"u}hring, Ingo and Raslan, Mones},
	journal={Neural Networks},
	volume={134},
	pages={107--130},
	year={2021},
	publisher={Elsevier}
}

@article{de2021approximation,
	title={On the approximation of functions by tanh neural networks},
	author={De Ryck, Tim and Lanthaler, Samuel and Mishra, Siddhartha},
	journal={Neural Networks},
	volume={143},
	pages={732--750},
	year={2021},
	publisher={Elsevier}
}

@article{lu2021deep,
	title={Deep network approximation for smooth functions},
	author={Lu, Jianfeng and Shen, Zuowei and Yang, Haizhao and Zhang, Shijun},
	journal={SIAM Journal on Mathematical Analysis},
	volume={53},
	number={5},
	pages={5465--5506},
	year={2021},
	publisher={SIAM}
}

@article{ellacott1994aspects,
	title={Aspects of the numerical analysis of neural networks},
	author={Ellacott, SW},
	journal={Acta Numerica},
	volume={3},
	pages={145--202},
	year={1994},
	publisher={Cambridge University Press}
}

@article{klusowski2016uniform,
	title={Uniform approximation by neural networks activated by first and second order ridge splines},
	author={Klusowski, Jason M and Barron, Andrew R},
	journal={arXiv preprint arXiv:1607.07819},
	year={2016}
}

@article{ma2022barron,
	title={The Barron space and the flow-induced function spaces for neural network models},
	author={E, Weinan and Ma, Chao and Wu, Lei},
	journal={Constructive Approximation},
	volume={55},
	number={1},
	pages={369--406},
	year={2022},
	publisher={Springer}
}

@article{ma2018priori,
	title={A priori estimates of the population risk for two-layer neural networks},
	author={E, Weinan and Ma, Chao and Wu, Lei},
	journal={Communications in Mathematical Sciences},
	volume={17},
	number={5},
	pages={1407--1425},
	year={2019},
	publisher={International Press of Boston}
}

@article{siegel2024sharp,
	title={Sharp bounds on the approximation rates, metric entropy, and n-widths of shallow neural networks},
	author={Siegel, Jonathan W and Xu, Jinchao},
	journal={Foundations of Computational Mathematics},
	volume={24},
	number={2},
	pages={481--537},
	year={2024},
	publisher={Springer}
}

@article{jagtap2020adaptive,
	title={Adaptive activation functions accelerate convergence in deep and physics-informed neural networks},
	author={Jagtap, Ameya D and Kawaguchi, Kenji and Karniadakis, George Em},
	journal={Journal of Computational Physics},
	volume={404},
	pages={109136},
	year={2020},
	publisher={Elsevier}
}

@article{wang2021understanding,
	title={Understanding and mitigating gradient flow pathologies in physics-informed neural networks},
	author={Wang, Sifan and Teng, Yujun and Perdikaris, Paris},
	journal={SIAM Journal on Scientific Computing},
	volume={43},
	number={5},
	pages={A3055--A3081},
	year={2021},
	publisher={SIAM}
}

@article{wang2026gradient,
	title={Gradient alignment in physics-informed neural networks: A second-order optimization perspective},
	author={Wang, Sifan and Li, Bowen and Perdikaris, Paris and others},
	journal={Advances in Neural Information Processing Systems},
	volume={38},
	pages={168482--168532},
	year={2026}
}

@article{bi2025extended,
	title={Extended Interface Physics-Informed Neural Networks Method for Moving Interface Problems},
	author={Bi, Ran and Deng, Weibing and Zhu, Yameng},
	journal={arXiv preprint arXiv:2508.01463},
	year={2025}
}

@article{dwivedi2020physics,
	title={Physics informed extreme learning machine (pielm)--a rapid method for the numerical solution of partial differential equations},
	author={Dwivedi, Vikas and Srinivasan, Balaji},
	journal={Neurocomputing},
	volume={391},
	pages={96--118},
	year={2020},
	publisher={Elsevier}
}

@article{gonon2023random,
	title={Random feature neural networks learn Black-Scholes type PDEs without curse of dimensionality},
	author={Gonon, Lukas},
	journal={Journal of Machine Learning Research},
	volume={24},
	number={189},
	pages={1--51},
	year={2023}
}

@article{de2025approximation,
	title={Approximation theory and applications of randomized neural networks for solving high-dimensional pdes},
	author={De Ryck, T and Mishra, S and Shang, Y and Wang, F},
	journal={arXiv preprint arXiv:2501.12145},
	year={2025}
}

@article{neufeld2023universal,
	title={Universal approximation property of Banach space-valued random feature models including random neural networks},
	author={Neufeld, Ariel and Schmocker, Philipp},
	journal={The Annals of Applied Probability},
	year={2024},
	publisher={Institute of Mathematical Statistics}
}

@article{wang2022and,
	title={When and why PINNs fail to train: A neural tangent kernel perspective},
	author={Wang, Sifan and Yu, Xinling and Perdikaris, Paris},
	journal={Journal of Computational Physics},
	volume={449},
	pages={110768},
	year={2022},
	publisher={Elsevier}
}

@inproceedings{rathore2024challenges,
	title={Challenges in training pinns: A loss landscape perspective},
	author={Rathore, Pratik and Lei, Weimu and Frangella, Zachary and Lu, Lu and Udell, Madeleine},
	booktitle={Forty-first International Conference on Machine Learning},
	year={2024}
}

@article{hu2022discontinuity,
	title={A discontinuity capturing shallow neural network for elliptic interface problems},
	author={Hu, Wei-Fan and Lin, Te-Sheng and Lai, Ming-Chih},
	journal={Journal of Computational Physics},
	volume={469},
	pages={111576},
	year={2022},
	publisher={Elsevier}
}

@article{dorfler1996convergent,
	title={A convergent adaptive algorithm for Poisson’s equation},
	author={D{\"o}rfler, Willy},
	journal={SIAM Journal on Numerical Analysis},
	volume={33},
	number={3},
	pages={1106--1124},
	year={1996},
	publisher={SIAM}
}

@article{cascon2008quasi,
	title={Quasi-optimal convergence rate for an adaptive finite element method},
	author={Cascon, J Manuel and Kreuzer, Christian and Nochetto, Ricardo H and Siebert, Kunibert G},
	journal={SIAM Journal on Numerical Analysis},
	volume={46},
	number={5},
	pages={2524--2550},
	year={2008},
	publisher={SIAM}
}

@article{karakashian2007convergence,
	title={Convergence of adaptive discontinuous Galerkin approximations of second-order elliptic problems},
	author={Karakashian, Ohannes A and Pascal, Frederic},
	journal={SIAM Journal on Numerical Analysis},
	volume={45},
	number={2},
	pages={641--665},
	year={2007},
	publisher={SIAM}
}

@article{liu2025integral,
	title={Integral Representations of Sobolev Spaces via ReLUk Activation Function and Optimal Error Estimates for Linearized Networks},
	author={Liu, Xinliang and Mao, Tong and Xu, Jinchao},
	journal={arXiv preprint arXiv:2505.00351},
	year={2025}
}

@article{zhu2025two,
	title={A Two-stage Adaptive Lifting PINN Framework for Solving Viscous Approximations to Hyperbolic Conservation Laws},
	author={Zhu, Yameng and Deng, Weibing and Bi, Ran},
	journal={arXiv preprint arXiv:2511.04490},
	year={2025}
}

@misc{driscoll2014chebfun,
	title={Chebfun guide},
	author={Driscoll, Tobin A and Hale, Nicholas and Trefethen, Lloyd N},
	year={2014},
	publisher={Pafnuty Publications, Oxford}
}

@article{cox2002exponential,
	title={Exponential time differencing for stiff systems},
	author={Cox, Steven M and Matthews, Paul C},
	journal={Journal of Computational Physics},
	volume={176},
	number={2},
	pages={430--455},
	year={2002},
	publisher={Elsevier}
}

@article{de2026least,
	title={Least squares with equality constraints extreme learning machines for the resolution of PDEs},
	author={De Falco, Davide Elia and Schiassi, Enrico and Calabr{\`o}, Francesco},
	journal={Journal of Computational Physics},
	volume={547},
	pages={114553},
	year={2026},
	publisher={Elsevier}
}

@article{zhang2024transferable,
	title={Transferable neural networks for partial differential equations},
	author={Zhang, Zezhong and Bao, Feng and Ju, Lili and Zhang, Guannan},
	journal={Journal of Scientific Computing},
	volume={99},
	number={1},
	pages={2},
	year={2024},
	publisher={Springer}
}

\appendix 
\section{Proof of Lemma~\ref{lemma 3}} \label{appendix A}
\begin{proof}
	By the Fubini-Tonelli theorem, we have
	\begin{equation}\label{eq A1}
		\begin{aligned}
			&\mathbb{E} \left[ \left\| \mathbb{E} \left[ X_1 \right] - \dfrac{1}{N} \sum_{i=1}^N X_i \right\|_{L^p(\Omega)}^p \right]
			=\int_\Omega \mathbb{E} \left[ \left| \mathbb{E} \left[ X_1 \right] - \dfrac{1}{N} \sum_{i=1}^N X_i\right|^p\right] \mathrm{d}\mu(x) \\
			&\qquad\qquad\quad=N^{-p} \int_\Omega \mathbb{E} \left[ \left| \sum_{i=1}^N \left(  \mathbb{E} \left[ X_1 \right] - X_i \right) \right|^p \right] \mathrm{d}\mu(x).
		\end{aligned}
	\end{equation}
	Applying the Marcinkiewicz-Zygmund inequality, we obtain
	\begin{equation}\label{ieq A2}
		\mathbb{E} \left[ \left| \sum_{i=1}^N \left(  \mathbb{E} \left[ X_1 \right] - X_i \right) \right|^p \right] 
		\le C_p \mathbb{E}\left[ \left( \sum_{i=1}^N \left| \mathbb{E} \left[ X_1 \right] - X_i \right| ^2 \right)^{p/2} \right].
	\end{equation}
	Consequently, it follows from the Minkowski's inequality in the probability space and the fact that $X_i$, $i=1,\dots,N$, are i.i.d.\ that
	{
		\begin{equation*}
			\begin{aligned}
				\left( \mathbb{E}\left[ \left( \sum_{i=1}^N \left|  \mathbb{E} \left[ X_1 \right] - X_i \right| ^2 \right)^{p/2} \right]\right) ^{2/p}
				&\le \sum_{i=1}^N \left( \mathbb{E}\left[ \left( \left|   \mathbb{E} \left[ X_1 \right] - X_i \right| ^2 \right)^{p/2} \right]\right) ^{2/p}\\
				&= N \left( \mathbb{E}\left[ \left|  \mathbb{E} \left[ X_1 \right] - X_1 \right| ^p \right]\right) ^{2/p},
			\end{aligned}
		\end{equation*}
	}
	which, when combined with \eqref{eq A1} and \eqref{ieq A2}, yields
	\begin{equation} \label{ieq A3}
		\mathbb{E} \left[ \left\| \mathbb{E} \left[ X_1 \right] - \dfrac{1}{N} \mathop{\sum}_{i=1}^N X_i \right\|_{L^p(\Omega)}^p \right] \le \dfrac{ C_p}{N^{p/2}} \left( \int_\Omega \mathbb{E} \left[ \left| \mathbb{E}\left[ X_1\right]  - X_1\right| ^p\right] \mathrm{d}\mu(x) \right). 
	\end{equation}
	Further, via Jensen's inequality we obtain
	\begin{equation*}
		\begin{aligned}
			\mathbb{E} \left[ \left\| \mathbb{E} \left[ X_1 \right] - \dfrac{1}{N} \sum_{i=1}^N X_i \right\|_{L^p(\Omega)} \right]
			&\le \left( \mathbb{E} \left[ \left\| \mathbb{E} \left[ X_1 \right] - \dfrac{1}{N} \sum_{i=1}^N X_i \right\|_{L^p(\Omega)}^p \right]\right)^{1/p} \\
			&\le \dfrac{C_p^{1/p}}{N^{1/2}} \left( \int_\Omega \mathbb{E} \left[ \left| \mathbb{E}\left[ X_1\right]  - X_1\right| ^p\right]  \mathrm{d}\mu(x) \right)^{1/p}.
		\end{aligned}
	\end{equation*}
	This completes the proof.
\end{proof}

\section{Proof of Theorem~\ref{Theorem 2}}\label{appendix B}

\begin{proof}
	From Lemma~\ref{lemma 2}, we have
	\begin{equation*}
		\left\| \partial_x^\alpha \left( u - u^M \right) \right\|_{L^p(\Omega)}
		\lesssim M^{-\eta} \|u\|_{\mathcal{B}_{|\alpha|+m}^{1,p}(\Omega)}.
	\end{equation*}
	Denote $S_M = G_\xi(M) \times G_s(M)$ and partition $S_M$ into mutually disjoint convex (or uniformly
	shape-regular Lipschitz) subdomains $\{S_i\}_{i=1}^N$ satisfying
	\[
	|S_i|\simeq \frac{|S_M|}{N},
	\qquad
	\operatorname{diam}(S_i)
	\lesssim \frac{M}{N^{1/(d+1)}}.
	\]
	Using this partition, we can rewrite the truncated function $u^M$ from \eqref{eq 3.4} as
	\begin{equation} \label{eq B1}
		\begin{aligned}
			u^M(x) 
			&= \sum_{i=1}^N \int_{S_i} \frac{|S_i|}{|S_i|\,|S_M|} f(\xi, s) \sigma(\xi \cdot x + s) \,\mathrm{d}\xi \,\mathrm{d}s \\
			&= \sum_{i=1}^N \widetilde{\pi}(S_i) \, \mathbb{E}_{\widetilde{\pi}_i} \left[ f(\xi, s) \sigma(\xi \cdot x + s) \right],
		\end{aligned}
	\end{equation}
	where $\widetilde{\pi}_i$ denotes the uniform probability measure on $S_i$, i.e.,
	\[
	\mathrm{d} \widetilde{\pi}_i(\xi, s) = \pi_i(\xi, s) \,\mathrm{d}\xi \,\mathrm{d}s = \frac{1}{|S_i|} \,\mathrm{d}\xi \,\mathrm{d}s.
	\]
	Let $c_i = \lceil \widetilde{\pi}(S_i) N \rceil$ be the number of samples drawn from $S_i$. Then
	\[
	\sum_{i=1}^{N} c_i = \mathcal{O}(N) \quad \text{and} \quad c_i \ge 1
	\]
	since $\sum_i \widetilde{\pi}(S_i) = 1$ and the partition $\{ S_i \}_{i=1}^N$ is quasi-uniform.
	We now construct a RaNN function based on this stratified sampling strategy:
	\begin{equation*}
		\begin{aligned}
			\widehat{U_W^{A,B}}(x) 
			&= \sum_{i=1}^{N} \widetilde{\pi}(S_i) \frac{1}{c_i} \sum_{j=1}^{c_i} f(A_{i,j}, B_{i,j}) \, \sigma(A_{i,j} \cdot x + B_{i,j}) \\
			&= \sum_{i=1}^{N} \sum_{j=1}^{c_i} \frac{\widetilde{\pi}(S_i)}{c_i} f(A_{i,j}, B_{i,j}) \, \sigma(A_{i,j} \cdot x + B_{i,j}) \\
			&=: \sum_{i=1}^{N} \sum_{j=1}^{c_i} W_{i,j} \, \sigma(A_{i,j}\cdot x + B_{i,j}),
		\end{aligned}
	\end{equation*}
	where, for each $S_i$, the parameters $\{ (A_{i,j}, B_{i,j}) \}_{j=1}^{c_i}$ are independently and identically distributed according to the uniform distribution over $S_i$. Define the corresponding random variables
	$
	X_{i, j}(x) = f(A_{i,j}, B_{i,j}) \, \sigma(A_{i,j} \cdot x + B_{i,j}).
	$
	The samples associated with different strata are generated independently; hence, the corresponding random fields are mutually independent. For any multi-index $\alpha$ with $|\alpha| \le k$, using \eqref{eq B1} and \eqref{ieq A3} we obtain
	{
		\begin{equation} \label{ieq B2}
			\begin{aligned}
				&\mathbb{E} \left[ \left\| \partial_x^\alpha u^M - \partial_x^\alpha \widehat{U_{W}^{A,B}} \right\|_{L^p(\Omega)}^p \right] \\
				&\lesssim\int_\Omega \left[\sum_{i=1}^{N} \dfrac{\widetilde{\pi}(S_i)^2}{c_i} \left( \mathbb{E}_{\widetilde{\pi}_i} \left[ \left| \mathbb{E}_{\widetilde{\pi}_i}(\partial_x^\alpha X_{i,1}) - \partial_x^\alpha X_{i,1} \right|^p \right]\right)^{2/p}  \right]^{p/2}  \\
				&\lesssim N^{p/2-1}\sum_{i=1}^{N} \widetilde{\pi}(S_i)^p \left( \frac{1}{c_i} \right)^{\frac{p}{2}} \biggl( \int_\Omega \mathbb{E}_{\widetilde{\pi}_i} \left[ \left| \mathbb{E}_{\widetilde{\pi}_i}(\partial_x^\alpha X_{i,1}) - \partial_x^\alpha X_{i,1} \right|^p \right] \,\mathrm{d}x \biggr).
			\end{aligned}
		\end{equation}
	}
	Here, the last step utilizes the algebraic inequality
	$$
	\left( \sum_{i=1}^N a_i \right)^{p/2} \le N^{p/2-1} \sum_{i=1}^N a_i^{p/2}, \quad p \ge 2.
	$$
	
	Define $
	\mathbb{V}_{i}^{\alpha, p} := \mathbb{E}_{\widetilde{\pi}_i} \left[ \left| \mathbb{E}_{\widetilde{\pi}_i}(\partial_x^\alpha X_{i,1}) - \partial_x^\alpha X_{i,1} \right|^p \right].
	$
	By the local Poincaré inequality \cite{evans2022partial}, we have
	\begin{equation*}
		\begin{aligned}
			\mathbb{V}_{i}^{\alpha, p}
			&= \frac{1}{|S_i|} \int_{S_i} \left| \partial_x^\alpha X_{i,1} - \frac{1}{|S_i|} \int_{S_i} \partial_x^\alpha X_{i,1} \,\mathrm{d}\xi \,\mathrm{d}s \right|^p \,\mathrm{d}\xi \,\mathrm{d}s \\
			&\lesssim \frac{1}{|S_i|} \left( \frac{M}{N^{\frac{1}{d+1}}} \right)^p \int_{S_i} \left| \partial_{(\xi, s)} \partial_x^\alpha X_{i,1} \right|^p \,\mathrm{d}\xi \,\mathrm{d}s.
		\end{aligned}
	\end{equation*}
	Consequently,
	\begin{equation*}
		\begin{aligned}
			&\mathbb{E} \left[ \left\| \partial_x^\alpha u^M - \partial_x^\alpha \widehat{U_{W}^{A,B}} \right\|_{L^p(\Omega)}^p \right] \\
			&\lesssim N^{p/2-1}\sum_{i=1}^N \widetilde{\pi}(S_i)^p \left( \frac{1}{c_i} \right)^{\frac{p}{2}} \frac{1}{|S_i|} \left( \frac{M}{N^{\frac{1}{d+1}}} \right)^p \int_\Omega \int_{S_i} \left| \partial_{(\xi, s)} \partial_x^\alpha X_{i,1} \right|^p \,\mathrm{d}\xi \,\mathrm{d}s \,\mathrm{d}x \\
			&\lesssim {\frac{M^{p - (d+1)}}{N^{\frac{p}{2} + \frac{p}{d+1}}}} \int_\Omega \int_{S_M} \left| \partial_{(\xi, s)} \partial_x^\alpha \left[ f(\xi, s) \sigma(\xi \cdot x + s) \right] \right|^p \,\mathrm{d}\xi \,\mathrm{d}s \,\mathrm{d}x.
		\end{aligned}
	\end{equation*}
	From the definition of $f(\xi, s)$ in Theorem~\ref{Theorem 1}, we have
	\begin{equation*}
		\begin{aligned}
			&\left| \partial_{(\xi, s)} \partial_x^\alpha \left[ f(\xi, s) \sigma(\xi \cdot x + s) \right] \right| \\
			&= \left| \partial_{(\xi, s)} \left[ \left( \prod_{i=1}^d a^{-\alpha_i} (a \xi_i)^{\alpha_i} \right) \sigma^{(|\alpha|)}(\xi \cdot x + s) f(\xi, s) \right] \right| \\
			&\lesssim M^{d+1} \left[ (1 + |a\xi|)^{|\alpha|-1} \left| \sigma^{(|\alpha|)}(\xi \cdot x + s) \right| \, |\widehat{u_e}(a\xi)|\right.  \\
			&\qquad + (1 + |a\xi|)^{|\alpha|} \left| \left( \sigma^{(|\alpha|+1)}(\xi \cdot x + s) + \sigma^{(|\alpha|)}(\xi \cdot x + s)\right) \right| \, |\widehat{u_e}(a\xi)| \\
			&\qquad\left.  + (1 + |a\xi|)^{|\alpha|} \left| \sigma^{(|\alpha|)}(\xi \cdot x + s) \right| \, |\partial_\xi \widehat{u_e}(a\xi)| \right].
		\end{aligned}
	\end{equation*}
	Using Lemma~\ref{lemma 1} and the fact that $1<mp$, we integrate over $S_M$ to obtain
	\begin{equation*}
		\int_{S_M} \left| \partial_{(\xi, s)} \partial_x^\alpha \left[ f(\xi, s) \sigma(\xi \cdot x + s) \right] \right|^p \,\mathrm{d}\xi \,\mathrm{d}s \lesssim M^{p(d+1)} \|u\|_{\mathcal{B}_{|\alpha| + m}^{1, p}(\Omega)}^p.
	\end{equation*}
	Therefore, applying Jensen's inequality yields
	\begin{equation*}
		\begin{aligned}
			\mathbb{E} \left[ \left\| \partial_x^\alpha u^M - \partial_x^\alpha \widehat{U_{W}^{A,B}} \right\|_{L^p(\Omega)} \right] 
			\lesssim \frac{M^{(1-1/p)(d+1)+1}}{{N^{1/2 + 1/(d+1)}}} \|u\|_{\mathcal{B}_{|\alpha|+ m}^{1, p}(\Omega)}.
		\end{aligned}
	\end{equation*}
	Summing over all multi-indices with $|\alpha| \le k$ gives
	\begin{equation*}
		\mathbb{E}\left[ \left\| u^M - \widehat{U_{W}^{A,B}} \right\|_{W^{k,p}(\Omega)} \right] \lesssim \frac{M^{(1-1/p)(d+1)+1}}{{N^{1/2 + 1/(d+1)}}} \|u\|_{\mathcal{B}_{k + m}^{1, p}(\Omega)}.
	\end{equation*}
	Finally, applying the triangle inequality, we obtain
	\begin{equation*}
		\begin{aligned}
			\mathbb{E}\left[ \left\| u - \widehat{U_{W}^{A,B}} \right\|_{W^{k,p}(\Omega)} \right]
			&\le \| u - u^M \|_{W^{k,p}(\Omega)} + \mathbb{E}\left[ \left\| u^M - \widehat{U_{W}^{A,B}} \right\|_{W^{k,p}(\Omega)} \right] \\
			&\lesssim \left( M^{-\eta} + \frac{M^{(1-1/p)(d+1)+1}}{{N^{1/2 + 1/(d+1)}}} \right) \|u\|_{\mathcal{B}_{k + m}^{1, p}(\Omega)}.
		\end{aligned}
	\end{equation*}
	The optimal approximation is achieved by balancing the two terms, i.e., {choosing $
		M \asymp N^{\frac{p(d+3)}{2(d+1)[(p-1)(d+1)+p+p\eta]}},
		$}
	which yields the final estimate \eqref{ieq 3.13} immediately.
\end{proof}	

\end{document}